\documentclass{siamart171218}
\ifpdf
\hypersetup{
  pdftitle={Optimal Convergence Rates for Tikhonov Regularization in Besov Spaces},
  pdfauthor={F. Weidling, B. Sprung, and T. Hohage}
}
\fi

\usepackage{amsmath,amssymb,dsfont,mathtools}
\usepackage{todonotes}
\usepackage{color}
\usepackage{mathrsfs}

\renewcommand{\theenumi}{\roman{enumi}}

\usepackage{multirow}
\usepackage{graphics}
\usepackage{tikz,tikzscale} 

\usepackage{pgfplots} 
\pgfplotsset{compat=1.11}
 
\usetikzlibrary{%
   external,calc
}

\newsiamthm{assumption}{Assumption}
\newsiamremark{remark}{Remark}
\newsiamremark{example}{Example}

\newcommand{\Rset}{\ensuremath{\mathds R}}
\newcommand{\Rd}{\ensuremath{{\mathds R^d}}}
\newcommand{\Td}{\ensuremath{{\mathds T^d}}}

\newcommand{\Nset}{\ensuremath{\mathds N}}
\newcommand{\Zset}{\ensuremath{\mathds Z}}

\newcommand{\dd}{\ensuremath{\mathrm d}}
\newcommand{\ee}{\ensuremath{\mathrm e}}
\newcommand{\ii}{\ensuremath{\mathrm i}}
\newcommand{\ftrue}{\ensuremath{{f^\dagger}}}
\newcommand{\gtrue}{\ensuremath{{g^\dagger}}}
\newcommand{\fsub}{\ensuremath{{f^*}}}
\newcommand{\fad}{\ensuremath{{\widehat{f}_\alpha}}}
\newcommand{\gobs}{\ensuremath{{g^{\mathrm{obs}}}}}

\newcommand{\Xspace}{\ensuremath{\mathcal X}}
\newcommand{\Yspace}{\ensuremath{\mathcal Y}}
\newcommand{\Rspace}{\ensuremath{\mathcal X}}
\newcommand{\Zspace}{\ensuremath{\mathcal Z}}
\newcommand{\Dset}{\ensuremath{\mathcal D}}
\newcommand{\Rpen}{\ensuremath{\mathcal R}}
\newcommand{\Cerr}{\ensuremath{C_\mathrm{err}}}
\newcommand{\BH}{\ensuremath{\mathrm{BH}}}
\newcommand{\js}{j}

\newcommand{\rand}{\ensuremath{Z}}

\DeclarePairedDelimiterX\altnorm[2]{\lVert}{\rVert}{#1\,\delimsize\vert\,\mathopen{}#2}
\DeclarePairedDelimiter\absval{\lvert}{\rvert}
\DeclarePairedDelimiter\norm{\lVert}{\rVert}
\DeclarePairedDelimiterX\pairing[2]{\langle}{\rangle}{#1,#2}
\DeclarePairedDelimiter\rbra{(}{)}
\DeclarePairedDelimiter\sbra{[}{]}
\DeclarePairedDelimiter\cbra{\{}{\}}

\DeclareMathOperator*{\argmin}{arg\,min}

\DeclareMathOperator{\ran}{ran}
\DeclareMathOperator{\supp}{supp}
\DeclareMathOperator{\err}{\mathbf{err}}
\DeclareMathOperator{\E}{\mathbb{E}}
\DeclareMathOperator{\prob}{\mathbb{P}}

\newcommand{\Besovpq}[1][s]{\ensuremath{B_{p,q}^{#1}}}
\newcommand{\Besovpqdual}[1][u]{\ensuremath{B_{p^\prime,q^\prime}^{#1}}}

\usepackage{color}

\headers{Tikhonov regularization in Besov spaces}{F. Weidling, B. Sprung and T. Hohage}

\title{{Optimal Convergence Rates for Tikhonov Regularization in Besov Spaces}\thanks{
\funding{This work was supported by Deutsche Forschungsgemeinschaft DFG through Project C09
of SFB 755 and Project B01 of RTG 2088.}}}

\author{Frederic Weidling\thanks{Institut f\"ur Numerische und Angewandte Mathematik, Universit\"at G\"ottingen, Lotze\-stra\ss{}e 16-18, 37083 G\"ottingen, Germany    (\email{f.weidling@math.uni-goettingen.de}, \email{b.sprung@math.uni-goettingen.de}, \email{hohage@math.uni-goettingen.de})}
\and
Benjamin Sprung\footnotemark[2]
\and
Thorsten Hohage\footnotemark[2]
}

\begin{document}
\maketitle

\begin{abstract}
This paper deals with Tikhonov regularization for linear and nonlinear ill-posed 
operator equations with wavelet Besov norm penalties.  We show order optimal 
rates of convergence for finitely smoothing operators and for the backwards 
heat equation for a range of Besov spaces using variational source conditions. 
We also derive order optimal rates for a white noise model 
with the help of variational source conditions and concentration inequalities 
for sharp negative Besov norms of the noise. 
\end{abstract}



\section{Introduction}
We consider ill-posed operator equations 
\begin{equation*}
	F(\ftrue)= \gtrue
\end{equation*}
with noisy right hand side and a forward operator $F\colon \Dset\subset \Xspace\to \Yspace$ 
where $\Xspace$ is some Besov space $B^s_{p,q}$, $\Dset\subset\Xspace$ a non-empty, closed, convex 
set and $\Yspace$ an $L^2$ space. $\ftrue\in \Dset$ denotes the true solution. 
We treat two noise models, the standard deterministic error model where
the observed data $\gobs$ satisfy
\begin{equation}\label{eq:det_noise}
	\gobs= \gtrue+\xi, \qquad \altnorm*{\xi}{\Yspace}\leq\delta 
\end{equation}
with a deterministic noise level $\delta>0$, and statistical models 
\begin{equation}\label{eq:rand_noise}
	\gobs= \gtrue+\varepsilon \rand
\end{equation}
with a statistical noise level $\varepsilon>0$ and some noise process $\rand$
on $\Yspace$ with white noise as most prominent example. 
(As we will have to deal with norms involving many, sometimes nested indices, we use 
the notation $\altnorm{\cdot}{\Yspace}$ instead of $\norm{\cdot}_{\Yspace}$ throughout the paper.) 

One of the most common approaches to compute a stable approximation of $\ftrue$ 
given $\gobs$ is Tikhonov regularization of the form
\begin{equation}\label{eq:tikhonov_det}
	\fad \in \argmin_{f\in \Dset} \left[\tfrac{1}{t}\altnorm*{F(f)-\gobs}{\Yspace}^t + \alpha\Rpen(f)\right]
\end{equation}
with some $t\geq 1$ and a regularization parameter $\alpha>0$.
In this paper we study the case that $\Rpen(f)$ is given by 
a norm power of a Besov norm of $f$.  
Such penalties with wavelet 
Besov norms with small index $p$ are frequently used to enforce sparsity 
(see, e.g.~\cite{Daubechies2004,RR:10}). As white noise on a Hilbert space $\Yspace$ does not 
belong to $\Yspace$ with probability $1$, 
we use $t=2$ in this case, expand the square, and omit the term $\frac{1}{2}\altnorm{\gobs}{\Yspace}^2$, 
which has no influence on the minimizer. This yields 
\begin{equation}\label{eq:tikhonov}
	\fad \in \argmin_{f\in \Dset} \left[\tfrac{1}{2}\altnorm*{F(f)}{\Yspace}^2-\pairing*{\gobs}{F(f)} + \alpha\Rpen(f)
	\right].
\end{equation}
A main goal of regularization theory are bounds on the distance of regularized estimators 
of the true solution in terms of the noise level 
$\delta$ or $\varepsilon$, respectively. In the case that $\Xspace$ and $\Yspace$ are Hilbert spaces 
such error bounds can be obtained by spectral theory (see e.g.\ \cite{EHN:96} for the deterministic case 
and \cite{BHMR:07} for the stochastic case). Concerning convergence results for deterministic regularization 
with Besov norms we refer to \cite{Daubechies2004} for convergence without rates, to  \cite{LT:08} for 
situations in which the rate $O(\sqrt{\delta})$ based on the source condition 
$\partial\Rpen(\ftrue)\cap \ran(T^*)\neq \emptyset$ 
can be achieved (possibly with respect to negative Sobolev norms) and 
to \cite{RR:10} for situations in which the rate $O(\delta^{2/3})$ 
under the source condition $\partial\Rpen(\ftrue)\cap \ran(T^*T)\neq \emptyset$ 
occurs. 
For statistical inverse problems minimax optimal rates under Besov smoothness assumptions have 
been shown for methods based on wavelet shrinkage 
(see \cite{CHR:04,Donoho1995,KPPW:07,KMR:06}).
Variational regularization has the advantage that no assumptions on the operator are required, it even works for nonlinear operators. 

In the last decade it has become popular to formulate source conditions in the form 
of variational inequalities and to derive convergence rates for Tikhonov regularization 
under such variational source conditions 
\cite{Flemming:12,grasmair:10,HKPS:07,Scherzer_etal:09,SKHK:12}. 
Recently, the authors have developed a method for the verification of 
variational source conditions in Hilbert spaces under standard 
smoothness conditions. This method has been successfully applied to a number 
of interesting inverse problems, both linear and nonlinear 
\cite{HW:15,HW:17,SH:18,WH:17}.

In this paper we extend our technique for the verification of variational 
source conditions to a Banach space setting. In particular, this allows derive 
variational source conditions for a certain class of operators if the true 
solution belongs to some Besov space $B^s_{p,q}$. This leads to optimal convergence rates 
for $p\in (1,2]$ and $q\geq 2$. 
An important step is a new characterization of subgradient smoothness 
(see Theorem \ref{thm:subdiffSmoothness}). As a second main novelty, 
we introduce a new technique to treat white noise and other stochastic noise models 
in non-quadratic variational regularization in an optimal way 
(Theorem \ref{thm:whiteNoiseRate}). We obtain not only order optimal convergence rates 
of the expected error, 
but also an estimate of the distribution of the error in terms of the 
distribution of a negative Besov norm of the noise process. At least in the case of 
white noise concentration inequalities for such negative Besov norms are known. 

The remainder of this paper is organized as follows: In \S \ref{sec:rates} we recall the definition of variational source conditions and the correspoding deterministic rates 
rates. Moreover, we formulate our new technique to derive error bounds for statistical 
noise model in an abstract functional analytic setting. 
Then we introduce a strategy for the verification of variational source condition including 
a characterization of subgradient smoothness in \S \ref{sec:strategy}. In the following two sections 
\ref{sec:finite_smoothing} and \ref{sec:BH} 
we present our results on convergence rates for finitely smoothing operators 
and for the backwards heat equation, respectively. Moreover, we show some numerical 
results for finitely smoothing operators confirming our theoretical 
error bounds. The paper has two appendices, one collecting 
properties of Besov spaces used in this paper, and the other one giving 
details on our numerical experiments.


\section{Variational source conditions}\label{sec:rates}

\subsection{Basic definitions}
Variational source conditions and the general error bounds in this section 
will be formulated in terms of 
the \emph{Bregman distance} $\Delta_{\Rpen}(f,\ftrue)$ between 
$f,\ftrue\in \Xspace$ with respect to some convex functional 
$\Rpen\colon \Xspace\to (-\infty,\infty]$ defined by  
\begin{equation*}
	\Delta_{\Rpen}(f,\ftrue):=\Rpen(f)-\Rpen(\ftrue)-\pairing{\fsub}{f-\ftrue},
\end{equation*}
where for the rest of this paper
\begin{equation*}
	\fsub\in \partial \Rpen(\ftrue).
\end{equation*}
The use of Bregman distances in inverse problems was first proposed in 
\cite{BO:04,eggermont:93}. Here $\partial \Rpen(\ftrue)$ denotes the 
subdifferential of $\Rpen$ at $\ftrue$ (see \cite{ET:76}). 
In general $\Delta_{\Rpen}(f,\ftrue)$ depends on the choice of the
subgradient $\fsub\in \partial \Rpen(\ftrue)$, but in this paper we will only 
consider differentiable 
penalty functionals $\Rpen$ such that  $\partial \Rpen(\ftrue)$ is a 
singelton (see \cite[Prop.~I.5.3]{ET:76}).
In this case $\Delta_\Rpen$ is the second order Taylor reminder 
which can often be related to more familiar distance measures: 
\begin{example}\label{ex:bregmanNormBound}
	\begin{enumerate}
		\item Let $\Rspace$ be a Hilbert space, then choosing $\Rpen(f)=\frac{1}{2}\altnorm{f-f_0}{\Rspace}^2$ for some $f_0\in \Rspace$ one obtains that
		\(
			\Delta_\Rpen (f,\ftrue)=\frac{1}{2}\altnorm*{f-\ftrue}{\Rspace}^2.
		\)
		\item\label{ex:bregmanNormBound:convex} If $\Rspace$ is an  
$r$-convex Banach space, then by \cite[Lemma~2.7]{Bonesky2008} there 
exists a constant $C_\Rspace>0$ such that 
		\begin{equation} \label{eq:bregmanNormBound:convex}
			\frac{C_\Rspace}{r} \altnorm*{f-\ftrue}{\Rspace}^r\leq \Delta_{\frac{1}{r}\altnorm{\cdot}{\Rspace}^r}(f,\ftrue), \qquad f,\ftrue\in\Rspace.
		\end{equation}
	\end{enumerate}
\end{example}

A variational source condition as first proposed in \cite{HKPS:07} for 
$\psi=\sqrt{\cdot}$ is an abstract smoothness condition for $\ftrue$ 
of the form
\begin{equation}\label{eq:vsc}
\forall f\in \Dset \colon\quad 
	\pairing*{\fsub}{\ftrue-f}\leq \frac{1}{2} \Delta_{\Rpen}(f,\ftrue) +\psi\rbra*{\altnorm{F(\ftrue)-F(f)}{\Yspace}^t}.
\end{equation}
Here $\psi:[0,\infty)\to [0,\infty)$ is a concave index function, i.e.\ $\psi$ is concave, continuous, increasing, and $\psi(0)=0$. \S \ref{sec:strategy} is devoted to the 
interpretation of such conditions.


In \cite[Ass.~1]{Werner2012} the notion of an \emph{effective noise level}
$\err \colon \Yspace\to [0,\infty)$ was introduced. $\err(F(\fad))$ bounds the effect of data noise 
on $\fad$, see Proposition \ref{prop:vscRate} and \S~\ref{sec:conv_rand_noise} below.  
It is defined for some fixed $\Cerr\geq 1$ by   
\begin{align}\label{eq:defErr}
	\err(g):&=
	\mathcal{S}(\gtrue)-\mathcal{S}(g)
	+\frac{1}{\Cerr t}\altnorm*{g-\gtrue}{\Yspace}^t
\end{align}
with the data fidelity term $\mathcal{S}(g):=\frac{1}{t}\altnorm*{g-\gobs}{\Yspace}^t$ in case of 
\eqref{eq:tikhonov_det} and $\mathcal{S}(g):=\frac{1}{2}\altnorm*{g}{\Yspace}^2-\pairing*{\gobs}{g}$ 
in case of \eqref{eq:tikhonov} where we set $t:=2$. 
In the deterministic case we choose $\Cerr=2^{t-1}$ in \eqref{eq:defErr}, in which case one can show that 
$\err(g)\leq\overline{\err}:=\frac{2}{t}\delta^t$, see \cite[Ex.~3.1]{Werner2012}. For bounds of $\err$ 
in the case of Poisson and impulsive noise we refer to \cite{Hohage2014,Werner2012}, and for bounds 
of $\err$ for the noise model \eqref{eq:rand_noise} to \S \ref{sec:conv_rand_noise}.

\subsection{Convergence rates for deterministic errors}
For the convergence rate theorem we will assume that $\psi$ is concave and 
define for convenience $\varphi_\psi (\tau):=(-\psi)^*\left(-\frac{1}{\tau}\right)$ which governs the bias. Here $\psi^*$ denotes the Fenchel conjugate for a convex function $\psi$  given by
\(
	\psi^*(v)=\sup\nolimits_u \sbra*{\pairing{v}{u}-\psi(u)}.
\)
Some corresponding calculus rules can be found, e.g.\ in \cite{ET:76}.
\begin{example}\label{ex:conjugate}
	The most common examples of index functions $\psi$ appearing in variational source conditions are either of H\"older or logarithmic type. If one extends these functions by $\psi(\tau)=-\infty$ for $\tau<0$ one can calculate $\varphi_\psi$ and obtains (see e.g.\ \cite{Flemming:12}):
	\begin{subequations}
		\begin{align}\label{eq:conjugate:hoelder} 
		&\mbox{H\"older-type:} &&
					\psi(\tau)= \tau^\mu \quad \implies\quad  \varphi_\psi(\tau)= c_\mu \tau^\frac{\mu}{1-\mu}\\
	\label{eq:conjugate:logarithmic} 
	&\mbox{logarithmic:} &&
		\begin{aligned}
			\psi(\tau) &=(- \ln \tau)^{-p} (1+o(1)) \text{ as }\tau \rightarrow 0\\
			\implies \varphi_\psi(\tau)&=(- \ln \tau)^{-p} (1+o(1))\text{ as }\tau \rightarrow 0.
		\end{aligned}
		\end{align}
	\end{subequations}
\end{example}

Under the assumption of a variational source condition one can prove the following convergence rates 
in terms of the effective noise model: 

\begin{proposition}[{\cite[Thm.~2.3]{Hohage2014}}]\label{prop:vscRate}
Assume	the variational source condition  \eqref{eq:vsc} holds true and let  $\fad$ be 
a	global minimizer  of the Tikhonov functional in \eqref{eq:tikhonov_det}.
	\begin{enumerate}
		\item\label{prop:vscRate:ratePointwise} Then $\fad$ satisfies the following error bounds:
		\begin{subequations}
		\begin{align}
			\frac{1}{2} \Delta_\Rpen(\fad,\ftrue) &\leq \frac{\err(F(\fad))}{2\alpha}+\varphi_\psi(2\Cerr \alpha \label{eq:vscRate:ratePointwise}),\\
			\altnorm*{F( \fad)-\gtrue}{\Yspace}^2 &\leq 2 \Cerr \err(F(\fad))+4 \Cerr \alpha \varphi_\psi (4\Cerr \alpha).\label{eq:vscRate:imageConvergence}
		\end{align}
		\end{subequations}
		\item\label{prop:vscRate:rateGlobal} If there exists some constant $\overline{\err}\geq \err(F(f))$ for all 
		$f$, the infimum of the right hand side of \eqref{eq:vscRate:ratePointwise} with $\err(F(\fad))$ replaced by 
		$\overline \err$ is attained if and only if $\alpha=\bar\alpha$ where $\bar\alpha$ is a solution to
		\begin{subequations}
		\begin{align}\label{eq:vscRate:globalAlpha}
			\frac{-1}{2\Cerr \bar\alpha}\in \partial (-\psi)(2\Cerr \overline{\err}).
		\end{align}
		In this case one obtains the convergence rate
		\begin{align}\label{eq:vscRate:rateGlobal}
			\frac{1}{2} \Delta_\Rpen(\widehat{f}_{\bar{\alpha}},\ftrue)\leq \Cerr \psi(\overline{\err}).
		\end{align}
		\end{subequations}
	\end{enumerate}
\end{proposition}

\subsection{Convergence rates for random noise}\label{sec:conv_rand_noise}
From now on we assume that the reader has some basic knowledge about Besov spaces which can be found in Appendix \ref{sec:besovSpaces}.
In the following let $\Omega$ be either a bounded Lipschitz domain or the $d$-dimensional torus $\Td:=\Rd/\Zset^d$  and $\mathscr{D}'(\Omega)$ the space of distributions on $\Omega$. In this subsection we consider noise models \eqref{eq:rand_noise} 
with a random variable $\rand$ in $\mathscr{D}'(\Omega)$. To prove convergence rates we need a deviation inequality for $\rand$  which is given by the following assumption.
\begin{assumption}\label{ass:dev_ineq}
	Assume that for all $\tilde{p}\in (1,\infty)$ we have $\rand\in B^{-d/2}_{\tilde{p},\infty}(\Omega)$ almost surely and that there exist constants $C_{\rand},M_{\rand},\mu>0$, such that
	\begin{align*}
\forall t>0\colon\quad 
		\prob\left(\altnorm*{\rand}{B^{-d/2}_{\tilde{p},\infty}}>M_{\rand}+t\right)\le \exp(-C_{\rand}t^\mu).
	\end{align*}
\end{assumption}
 If $\rand=W$ is a Gaussian white noise  and $\Omega=\Td$, then  Assumption \ref{ass:dev_ineq} holds true with $\mu=2$, 
$M_{\rand}$ being the median of $\altnorm{W}{B^{-d/2}_{\tilde{p}^\prime,\infty}}$ and $C_{\rand}$ depending on $\tilde{p}$ and $d$ 
(see \cite[Thm.~3.4, Cor.~3.7]{Veraar2011} or \cite[remark after Thm.~4.4.3]{Nickl2016}). 
In the following let $p^\prime$ denote the H\"older conjugate for any  number $1 \leq p \leq \infty$.


For random noise we choose $\Cerr=1$ in \eqref{eq:defErr} and obtain $\err(g)= \varepsilon\pairing{\rand}{g-\gtrue}$, hence it is natural to estimate the error functional for $p\in (1,2]$ via
\begin{equation}\label{eq:effective_noise_rand}
	\err\rbra*{g} =\varepsilon\pairing*{\rand}{g-\gtrue} \leq \varepsilon\altnorm*{\rand}{B^{-d/2}_{p^\prime,\infty}} \altnorm*{g-\gtrue}{B^{d/2}_{p,1}}
\end{equation}
and the challenge is to find a good control of the second factor. We will show that this can be done via an interpolation approach which for our two model problems results again in optimal rates (for $q\geq 2$).
To formulate a general error bound, we will assume for the moment that the second factor can be estimated 
as follows:
\begin{assumption}\label{ass:interpolation}
There exist constants $C,\beta, \gamma>0$ such that the inequality
	\begin{equation}\label{eq:interpolation}
		\altnorm*{F(f_1)-F(f_2)}{B^{d/2}_{p,1}}\leq C \altnorm*{F(f_1)-F(f_2)}{L^2}^\beta \Delta_\Rpen (f_1,f_2)^\gamma
	\end{equation}
	holds true for all $f_1,f_2\in \Rspace$.
\end{assumption}
Often this assumption can be verified by Remark \ref{rem:interpolationBregmanBound} below.

\begin{theorem}\label{thm:whiteNoiseRate}
	Let a variational source condition \eqref{eq:vsc} and Assumption \ref{ass:interpolation} be fulfilled, 
	and let $\fad$ be a global minimizer of the Tikhonov functional in \eqref{eq:tikhonov}.
	\begin{enumerate}
		\item\label{thm:whiteNoiseRate:error} If $0<\beta<2$, then the effective noise level at $F(\fad)$ is bounded by
		\begin{equation*}
			\err\rbra*{F(\fad) }\leq C \altnorm*{\varepsilon\rand}{B^{-d/2}_{p^\prime,\infty}}^{\frac{2}{2-\beta}} \Delta_\Rpen(f,\ftrue)^{\frac{2\gamma}{2-\beta}} +2\alpha \varphi_\psi(4\alpha).
		\end{equation*}
		\item\label{thm:whiteNoiseRate:rate} If in addition $0<\gamma<\frac{1}{2}(2-\beta)$, this implies the error bound
		\begin{equation*}
			\frac{1}{2} \Delta_\Rpen(\fad,\ftrue) \leq C \alpha^{-\frac{2-\beta}{2-\beta-2\gamma} } \altnorm*{\varepsilon\rand}{B^{-d/2}_{p^\prime,\infty}}^{\frac{ 2}{(2-\beta)-2\gamma}}+ 4 \varphi_\psi ( 4\alpha ).
		\end{equation*}
	\end{enumerate}
\end{theorem}
\begin{proof}
	For \eqref{thm:whiteNoiseRate:error} note that due to Assumption \ref{ass:interpolation} we obtain
	\begin{align*}
		\err\rbra*{F(\fad)} &=\pairing*{\varepsilon\rand}{F(\fad)-\gtrue} \leq \altnorm*{\varepsilon\rand}{B^{-d/2}_{p^\prime,\infty}} \altnorm*{F(\fad)-\gtrue}{B^{d/2}_{2,1}}\\
		&\leq C \altnorm*{\varepsilon\rand}{B^{-d/2}_{p^\prime,\infty}} \altnorm*{F(\fad)-\gtrue}{L^2}^\beta \Delta_\Rpen (\fad,\ftrue)^\gamma. 
	\end{align*}
	By the image space convergence rate \eqref{eq:vscRate:imageConvergence} of Proposition \ref{prop:vscRate} we can estimate 
	\begin{align*}
		\err\rbra*{F(\fad)} &\leq C \altnorm*{\varepsilon\rand}{B^{-d/2}_{p^\prime,\infty}}  \sbra[\big]{2 \err\rbra*{F(\fad)} +4 \alpha \varphi_\psi(4\alpha)}^{\beta/2} \Delta_\Rpen(\fad,\ftrue)^{\gamma} \\
		&\leq C \altnorm*{\varepsilon\rand}{B^{-d/2}_{p^\prime,\infty}}^{\frac{2}{2-\beta}} \Delta_\Rpen(\fad,\ftrue)^{ \frac{2\gamma}{2-\beta}} + \frac{1}{2}\sbra[\big]{ \err\rbra*{F(\fad)} +2\alpha \varphi_\psi(4\alpha)}
	\end{align*}
	by Young's inequality, with a generic constant $C>0$. Rearranging terms yields the bound on the effective noise level.
	
	To prove (\ref{thm:whiteNoiseRate:rate}), note that due to eq.~\eqref{eq:vscRate:ratePointwise} in Proposition \ref{prop:vscRate} we have
	\begin{align*}
		\Delta_\mathcal{R}(\fad,\ftrue)\le \frac{ \err(F(\fad))}{\alpha}+2\varphi_\psi(2\alpha)\le \frac{ \err(F(\fad))}{\alpha}+2\varphi_\psi(4\alpha).
	\end{align*}
	Together with the first part we obtain
	\begin{align*}
		\Delta_\mathcal{R}(\fad,\ftrue)&\leq C \alpha^{-1}\Delta_\Rpen(\fad,\ftrue)^{ \frac{2\gamma}{2-\beta}}\altnorm*{\varepsilon\rand}{B^{-d/2}_{p^\prime,\infty}}^{\frac{2}{2-\beta}}+ 4 \varphi_\psi(4\alpha)\\
		&\leq \frac{1}{2} \Delta_\Rpen(\fad,\ftrue) + C \alpha^{-\frac{2-\beta}{2-\beta-2\gamma} } \altnorm*{\varepsilon\rand}{B^{-d/2}_{p^\prime,\infty}}^{\frac{2}{2-\beta-2\gamma}}+ 4 \varphi_\psi ( 4\alpha )
	\end{align*}
	which proves the claim.
\end{proof}

\section{Verification of variational source conditions}\label{sec:strategy}
In \cite[Thm.~2.1]{HW:17} two of the authors formulated a strategy for the verification of 
variational source conditions in terms of orthogonal projection operators. In this section 
we will extend this strategy to Banach space settings. It turns out that the smoothness of subgradients 
of the solution rather than the smoothness of the solution itself determines the convergence rate. 
Therefore, a crucial step will be the analysis of the smoothness of subgradients in Besov spaces.

\subsection{Preliminaries}\label{sec:prelim}
One of the main difficulties when trying to prove a variational source condition is the Bregman distance 
appearing on the right hand side of \eqref{eq:vsc} since its properties depend very much on the specific choice 
of the regularization functional. Hence we are going to assume that it can be estimated from below by some 
norm power.
\begin{assumption}\label{ass:bregmanNormBound}
There exist constants $C_\Delta>0$ and $r>1$ such that 
	\begin{equation*}
		C_\Delta \altnorm*{f_1-f_2}{\Rspace}^r\leq \Delta_\Rpen (f_2,f_1)
\qquad \mbox{for all }f_1,f_2\in\Rspace.
	\end{equation*} 
\end{assumption}

This assumption is satisfied in particular if $\Xspace$ is convex of power type and $\Rpen$ is a 
norm power, see Example \ref{ex:bregmanNormBound}. 
However the case of $\Rpen(\cdot)=\altnorm{\cdot}{\ell^1}$ shows that this is not always the case, if e.g.\ $f_1,f_2>0$, then $\Delta_R(f_2,f_1)=0$, but $\altnorm{f_1-f_2}{\ell^1}$ might be arbitrary large.

\begin{remark}\label{rem:interpolationBregmanBound}
	Under Assumption \ref{ass:bregmanNormBound} eq.~\eqref{eq:interpolation} in Assumption \ref{ass:interpolation} is fulfilled if  
	\begin{equation}\label{eq:interpolation_norms}
		\altnorm*{F(f_1)-F(f_2)}{B^{d/2}_{p,1}}\leq C \altnorm*{F(f_1)-F(f_2)}{L^2}^\beta \altnorm*{f_1-f_2}{\Rspace}^{\gamma r}. 
	\end{equation}	
	Note that for linear operators $F$ one necessarily has $\beta+\gamma r=1$. 
	For a nonlinear operators $F$ Assumption \ref{ass:interpolation} can also be verified by standard interpolation 
inequalities if $F$ maps Lipschitz continuously into some space of higher regularity.
\end{remark}

\subsection{Basic strategy}
Our main tool for the derivation of variational source conditions will be the following generalization 
of Theorem~2.1 in \cite{HW:17}:

\begin{theorem}\label{thm:strategy}
Let $\Xspace$ and $\Yspace$ 
be Banach spaces and $\Rpen$ a penalty term such that Assumption \ref{ass:bregmanNormBound} is fulfilled. 
Let $\ftrue\in\Dset$ and $\fsub\in\partial \Rpen(\ftrue)$. Suppose that there exists a family of operators $P_j\colon \Rspace^*\rightarrow \Rspace^*$ for $j\in J$ an index set such that for some functions $\kappa,\sigma\colon J\rightarrow (0,\infty)$ and a constant $\gamma\geq 0$ the following holds true for all $j\in J$:
	\begin{subequations}\label{eq:strategy}
	\begin{align}
		&\altnorm*{(I-P_j) \fsub}{\Rspace^*}\leq \kappa(j) \label{eq:strategySmoothness}\\
		&\inf_{j\in J} \kappa(j)=0 \label{eq:strategyIndex}\\
		&\begin{aligned}\label{eq:strategyIllposed}
			&\pairing*{P_j \fsub}{\ftrue-f}\leq \sigma(j) \altnorm*{F(\ftrue)-F(f)}{\Yspace}+\gamma \kappa(j) \altnorm*{\ftrue-f}{\Rspace}\\
			&\qquad\text{for all $f\in \Dset$ with }\altnorm*{\ftrue-f}{\Rspace}\leq \rbra*{\tfrac{2}{C_\Delta} \altnorm*{\fsub}{\Rspace^*}}^{\frac{r^\prime}{r}}.
		\end{aligned}	
	\end{align}
	\end{subequations}
	Then $\ftrue$ fulfills a variational source condition \eqref{eq:vsc} with the concave index function
	\begin{equation}\label{eq:strategyPsi}
		\psi_\mathrm{vsc}(\tau)=\inf_{j\in J} \sbra*{\sigma(j)\tau^{1/t}+\frac{1}{r^\prime} \rbra*{\tfrac{2}{C_\Delta}}^{r^\prime/r} \rbra*{1+\gamma}^{r^\prime} \kappa(j)^{r^\prime}}.
	\end{equation}
\end{theorem}

Condition \eqref{eq:strategySmoothness} describes 
the smoothness of the solution (actually rather the smoothness of the subdifferential, but in the 
examples considered later one of the two uniquely determines the other, see Theorem \ref{thm:subdiffSmoothness}), 
whereas \eqref{eq:strategyIllposed} describes the local ill-posedness of the problem.

\begin{example}\label{ex:strategy}
	In order to illustrate these interpretations, consider the case that $\Xspace, \Yspace$ are Hilbert spaces, $F$ is an injective compact operator and let $(f_j,g_j,\sigma_j)_{j\in \Nset}$ be the corresponding singular system. Set  $P_j f= \sum_{k\leq j} \pairing{f}{f_j}f_j$. Then we obtain that $\fsub=\ftrue$ and  \eqref{eq:strategySmoothness} reads
	\begin{equation*}
		\altnorm*{(I-P_j)\ftrue}{\Xspace}=\rbra*{\sum\nolimits_{k>j} \absval{\pairing{\ftrue}{f_j}}^2}^{1/2}=:\kappa(j)
	\end{equation*}
	i.e.\ $\kappa$ measures the decay rate of the coefficients of $\ftrue$ in the system $(f_j)_{j\in \Nset}$. In the case where $f_j$ are trigonometric polynomials this measures classical smoothness. Denoting by $Q_j g = \sum_{k\leq j} \pairing{g}{g_j}g_j$ we obtain an inequality of the form \eqref{eq:strategyIllposed} via
	\begin{align*}
		\pairing{P_j \ftrue}{\ftrue-f} &\leq \altnorm*{P_j\ftrue}{\Xspace} \altnorm*{P_j(\ftrue-f)}{\Xspace} \leq \altnorm*{P_j\ftrue}{\Xspace} \frac{1}{\sigma_j} \altnorm*{Q_j T(\ftrue-f)}{\Yspace}\\
		&= \rbra*{\sum\nolimits_{k\leq j} \frac{\absval{\pairing{\ftrue}{f_j}}^2}{\sigma_j^2}}^{1/2}  \altnorm*{T(\ftrue-f)}{\Yspace}= \sigma(j)\altnorm*{T(\ftrue-f)}{\Yspace}
	\end{align*}
with $\sigma(j):= \rbra{\sum\nolimits_{k\leq j} \absval{\pairing{\ftrue}{f_j}}^2}^{1/2}/\sigma_j\leq \altnorm{\ftrue}{\Xspace}/\sigma_j$, 
	i.e.\ $\sigma$ measures the decay rate of the singular values of $F$ relative to the decay rate of the coefficients of $\ftrue$ in the singular system.
\end{example}

\begin{proof}[Proof of Theorem \ref{thm:strategy}:]
First assume that
$f$ does not satisfy the condition in the second line of 
\eqref{eq:strategyIllposed} or  equivalently that 
$\altnorm{\fsub}{\Rspace^*}\leq\tfrac{C_\Delta}{2}\altnorm{\ftrue-f}{\Rspace}^{r-1}$. Then
	\begin{align*}
		\pairing*{\fsub}{\ftrue-f}&\leq \altnorm{\fsub}{\Rspace^*}\altnorm{\ftrue-f}{\Rspace}\leq \tfrac{C_\Delta}{2}\altnorm{\ftrue-f}{\Rspace}^{r}\leq \tfrac{1}{2}\Delta_{\Rpen}(f,\ftrue),
	\end{align*}
	that is the variational source condition holds true even with $\psi\equiv 0$.	Otherwise using \eqref{eq:strategySmoothness}, \eqref{eq:strategyIllposed} and Young's inequality we get for each $j\in J$ that
	\begin{align*}\allowdisplaybreaks
		&\pairing*{\fsub}{\ftrue-f}\\
		=\,&\pairing*{P_j\fsub}{\ftrue-f} + \pairing*{(I-P_j)\fsub}{\ftrue-f}\\
		\leq\,& \sigma(j) \altnorm*{F(\ftrue)-F(f)}{\Yspace}+(1+\gamma) \kappa(j) \altnorm*{\ftrue-f}{\Rspace}\\
		\leq\,& \sigma(j) \altnorm*{F(\ftrue)-F(f)}{\Yspace}+\tfrac{1}{r^\prime} \rbra*{\tfrac{2}{C_\Delta}}^{r^\prime/r} \rbra*{1+\gamma}^{r^\prime} \kappa(j)^{r^\prime} + \tfrac{C_\Delta}{2r}\altnorm*{\ftrue-f}{\Rspace}^r\\
		\leq\,& \sigma(j) \altnorm*{F(\ftrue)-F(f)}{\Yspace}+\tfrac{1}{r^\prime} \rbra*{\tfrac{2}{C_\Delta}}^{r^\prime/r} \rbra*{1+\gamma}^{r^\prime} \kappa(j)^{r^\prime} + \tfrac{1}{2}\Delta_{\Rpen}(f,\ftrue).
	\end{align*}
	Taking the infimum over the right hand side with respect to $j\in J$ yields \eqref{eq:strategyPsi} with 
	$\tau=\altnorm{F(\ftrue)-F(f)}{\Yspace}^t$.
	
Note that $\psi$ is defined as an infimum over concave and increasing functions and hence is also increasing and concave. By \eqref{eq:strategyIndex} we obtain that $\psi(0)=0$, and hence $\psi$ is indeed an index function.
\end{proof}
For linear operators and under the conditions below one can choose $\gamma=0$, and the additional restriction 
that one needs \eqref{eq:strategyIllposed} only for $\altnorm*{\ftrue-f}{\Rspace}$ small is not necessary as already seen in the specific example. We have included both complications here since they may be needed  
for some non-linear operators, see e.g.\ \cite{HW:15}. 
We will present two differnt applications of the strategy in Sections \ref{sec:finite_smoothing} and \ref{sec:finite_smoothing} and discuss how they may generalize to similar setting in Remark \ref{rem:strategyLipschitz} and \ref{rem:strategyInterchange}.

\subsection{Besov spaces}\label{sec:besovSpacesGeneral}
Recall that we assume $\Omega$ to be either a bounded Lipschitz domain or $\Td$. We now introduce two equivalent norms on $\Besovpq$, further details on these spaces can be found in Appendix \ref{sec:besovSpaces}.
\subsubsection{Wavelet systems}
For any $j\in\Nset_0$ let $I_j$ be a countable index set, let 
$I:=\{(j,l):j\in \Nset_0, l\in I_j\}$, and let 
$(\phi_{j,l})_{(j,l)\in I }\subset L^2(\Omega)$ be an orthonormal system. 
Assume that for some $\sigma\in \Nset$ the system fulfills $\phi_{j,l} \in C^\sigma(\overline\Omega,\Rset)$.
For $f\in \mathcal(C^\sigma(\overline\Omega,\Rset))^\prime$ define 
the linear mapping 
\begin{subequations}\label{eq:waveletBesovGeneral}
\begin{equation}\label{eq:waveletBesovGeneral:transform}
	\mathcal W f := \lambda := (\lambda_{j,l})_{(j,l)\in I}\qquad \text{where} \qquad \lambda_{j,l}:=\int_{\Omega} f(x) \phi_{j,l}(x) \,\dd x.
\end{equation}
We assume that this system also generates Besov spaces in the following way:
For all $p,q\in[1,\infty]$ and $s\in \Rset$ with $\absval{s}<\sigma$ we get that $f\in \mathcal D^\prime(\Omega)$ belongs to $\Besovpq(\Omega)$ if and only if
\begin{align} \label{eq:waveletBesovGeneral:norm}
	\altnorm*{f}{\Besovpq}_{\mathcal{W}}:=\sbra*{\sum\nolimits_{j\in \Nset_0} 2^{jsq} 2^{jd(\frac{1}{2}-\frac{1}{p})q} \rbra*{\sum\nolimits_{l\in I_j} \absval{\lambda_{j,l}}^p }^{\frac{q}{p}} }^{\frac{1}{q}}
\end{align}
(with the usual modifications if $p=\infty$ or $q=\infty$) is finite. In this case \eqref{eq:waveletBesovGeneral:norm} is an equivalent norm on $\Besovpq[s]$. Furthermore this usually implies that 
\begin{align}\label{eq:waveletBesovGeneral:decompose}
	\text{the decomposition}\qquad f:=\mathcal W^* \lambda = \sum\nolimits_{(j,l)\in I}  \lambda_{j,l} \phi_{j,l} \qquad\text{ is unique} 
\end{align}
with unconditional convergence in $\mathcal D^\prime(\Omega)$ and local convergence in $\Besovpq[u]$ for all $u<s$ (even in $\Besovpq[s]$ if $p\neq \infty$ and $q\neq \infty$). 
\end{subequations}

\begin{example}\label{ex:waveletSystem} The properties above are satisfied  
in particular in the following cases:
	\begin{enumerate}
		\item Let $\Omega=\Td$, and let $(\tilde \phi_{j,l})_{(j,l)\in \tilde I }$ be either the $d$-dimensional Daubechies wavelet system of order $n\in \Nset$ or Meyer wavelet system for $\Rset^d$ where $\tilde I=\{(j,l):j\in\Nset_0,l\in \tilde I_j\}$ with $I_0:=\Zset^d$ and $I_j=\cbra{1,\ldots,2^d-1}\times\Zset^d$. 
	Define periodization $\phi_{j, l}(x) := \sum\nolimits_{z\in \Zset^d} \tilde\phi_{j, l} (x-z)$ for $x\in \Td$ of $\tilde\phi_{j,l}$ for $l\in I_j\subset \tilde I_j$ with $I_0:=\cbra{0}$ and $I_j:=\cbra{1,\ldots,2^d-1}\times\cbra{z\in \Zset^d \colon 0\leq z_i < 2^{j}}$. 
Then  $(\phi_{j,l})_{(j,l)\in I }$ is an orthonormal system in $L^2(\Td)$. 
Furthermore for fixed $\sigma\in \Nset$ the system fulfills \eqref{eq:waveletBesovGeneral} where in case of the Daubechies wavelet system we have to choose $n$ large enough (one has $\phi_{j,l}\in C^{1}$ for $n\geq 3$ and $\phi_{j,l}\in C^{2}$ for $n\geq 7$ while for large $n$ the asymptotic formula $\phi_{j,l}\in C^{\sigma}$ for $\sigma>0.2 n$ holds true, see \cite[Sec.~7]{Daubechies1992}).
		\item If $\Omega\subset \Rd$ is a bounded Lipschitz domain, there are different ways to define Besov spaces that differ mainly in how to treat boundary values.
	Let
	\begin{align*}
		B^s_{p,q}(\Omega):=
		\begin{cases}
			\{f\in \mathcal{D}^\prime(\Omega) : f=h|_\Omega, h\in B^s_{p,q}(\Rd)  \} &\text{if }s\le 0 \\
			\{f\in \mathcal{D}^\prime(\Omega) : f=h|_\Omega, h\in B^s_{p,q}(\Rd), \supp h\subset \overline{\Omega}\} &\text{if }s>0
		\end{cases}
	\end{align*}
	and 
	\begin{align*}
		\altnorm*{f}{B^{s}_{p,q}(\Omega)}:=\inf\altnorm*{h}{B^{s}_{p,q}(\Rd)},
	\end{align*}
	where the infimum is taken over all extensions $h$ as above. 
		For $B^s_{p,q}(\Omega)$ defined like this an explicit construction of an orthonormal system based on the Daubechies orthonormal wavelet system and a Whitney decomposition fulfilling \eqref{eq:waveletBesovGeneral} is carried out in \cite[Thm.~2.33 and 3.23]{Triebel2008}. For computationally more feasible constructions of orthogonal wavelets 
on the interval (or boxes via tensor product constructions) we refer to 
\cite{CDV:93} and with respect to boundary conditions to \cite{auscher:93}.
	\end{enumerate}
\end{example}
Motivated by the above example we will call $(\phi_{j,l})_{(j,l)\in I }$ a wavelet system and $\mathcal W$ the wavelet transform. In the following we will always assume that the wavelet system is such that \eqref{eq:waveletBesovGeneral:decompose} holds true and hence we have a norm on $\Besovpq(\Omega)$ given by \eqref{eq:waveletBesovGeneral:norm}.

\subsubsection{Fourier system}
In case that $\Omega=\Td$ an equivalent norm on $\Besovpq$ for $p\in (1,\infty)$, $q\in[1,\infty]$ and $s\in \Rset$ is given by the following:
Denote by $\chi_0(x)$ the characteristic function of the unit square in $\Rd$ and by the system $\rbra{\chi_j}_{j\in\Nset_0}$ the corresponding dyadic resolution of unity, that is
\begin{subequations}\label{eq:fourierBesovNorm}
\begin{equation}\label{eq:fourierBesovNorm:chi}
	\chi_0(x)=\begin{cases} 1 &\absval{x}_\infty \leq 1 \\ 0 & \text{else} \end{cases} \quad\text{and} \quad \chi_j(x):= \chi_0(2^{-j} x)- \chi_0(2^{-j+1} x)
\end{equation}
for $j\in \Nset$. Furthermore, we will denote by
\begin{equation*}
	\mathcal F f :=\hat f :=  \rbra*{\hat f(z)}_{z\in \Zset^d}\quad\text{where} \quad \hat f(z) := \int_{\Td} f(x) \overline{e_z(x)} \, \dd x\quad \text{with}\quad  e_z(x):=\ee^{2 \pi \ii x\cdot z}
\end{equation*}
the Fourier transform on $\Td$. For a function $f$ defined on  $\mathds T^d$ set
\begin{equation}\label{eq:fourierBesovNorm:innerSet}
	I_j:=\cbra{z\in \Zset^d \colon \chi_j(2 \pi z)=1}.
\end{equation}
A norm on $\Besovpq$ is then defined by
\begin{equation}\label{eq:fourierBesovNorm:norm}
	\altnorm*{f}{\Besovpq}:=\sbra*{\sum\nolimits_{j\in \Nset_0} 2^{jsq} \altnorm*{\sum\nolimits_{l\in I_j} \hat f(l) e_l }{L^p(\mathds T^d)}^q}^{\frac{1}{q}}
\end{equation}
\end{subequations}
(see \cite[Sec.~1.3]{Triebel2008}) with the usual modification for $q=\infty$. Note that while it is ``standard'' to introduce these spaces via a dyadic resolution of unity as we do above it is usually assumed that this resolution is also smooth (see e.g.\ \cite[Sec.~2.3]{Triebel2010}). However, this is not required for the range of the parameter $p\in (1,\infty)$ which we are considering here (see e.g.\ \cite[Sec.~2.5.4]{Triebel2010}).

\subsection{Subgradient smoothness}
If $\Rspace$ is not a Hilbert space, then the mapping $f\mapsto \fsub$ is no longer the identity mapping. While the continuity properties of this mapping have been studied for some time (see e.g.\ \cite{Chakrabarty2007} and references therein), much less is known
on the question whether additional smoothness of $f$ yields additional smoothness of $\fsub$. Although not stated explicitly in this form, the results in \cite{LT:08,RR:10} essentially show that for $B^{s_2}_{p_2,p_2}\subset B^{s_1}_{p_1,p_1}$  and $\fsub \in \partial \frac{1}{p_1} \altnorm{f}{B^{s_1}_{p_1,p_1}}^{p_1}_{\mathcal W}$ for a smooth enough wavelet system the relation
\begin{equation}\label{eq:subgradKnown}
	f\in B^{s_2}_{p_2,p_2} \quad \Longrightarrow \quad \fsub \in B^{-s_1}_{p_1^\prime,p_1^\prime} \cap B^{s_3}_{p_3,p_3}
\end{equation}
holds true where $p_3=\frac{p_2}{p_1-1}$ and $s_3=-s_1+(s_2-s_1)(p_1-1)$. The proof of \eqref{eq:subgradKnown} 
can be carried out along the lines of the proof of the next theorem, showing that it is even an equivalence result. For our new result, we restrict to the case that $p_1=p_2$ allowing for different fine indices $q_j$:

\begin{theorem}\label{thm:subdiffSmoothness}
	Let $p,q_1\in(1,\infty)$, $q_2\in[1,\infty]$, $s_1,s_2 \in \Rset$ and $r>0$. Set $s_3=-s_1+(s_2-s_1)(q_1-1)$ and $q_3=\frac{q_2}{q_1-1}$ and assume that the chosen wavelet system fulfills the assumptions in \S~\ref{sec:besovSpacesGeneral} with $\sigma>\max\cbra{|s_1|,|s_2|,|s_3|}$. 
	Let $\fsub \in \partial \frac{1}{r} \altnorm{f}{B^{s_1}_{p,q_1}}_{\mathcal W}^r$, then $f\in B^{s_1}_{p,q_1} \cap B^{s_2}_{p,q_2}$ if and only if $\fsub \in B^{-s_1}_{p^\prime,q_1^\prime} \cap B^{s_3}_{p^\prime,q_3}$ . Furthermore
	\begin{equation*}
		\altnorm*{\fsub}{B^{s_3}_{p^\prime,q_3}}_{\mathcal W}= \altnorm*{f}{B^{s_1}_{p,q_1}}_{\mathcal W}^{r-q_1} \altnorm*{f}{B^{s_2}_{p,q_2}}_{\mathcal W}^{q_1-1}.
	\end{equation*}
\end{theorem}
\begin{proof}
	We obtain $\fsub\in B^{-s_1}_{p^\prime,q_1^\prime}$ directly by the mapping properties of the subdifferential.
	
	Denote by $\lambda =\mathcal W f$ the wavelet decomposition of $f$. Then,  as for the given range of the parameters $p,q_1,r$ the norm is differentiable, one obtains that  $\fsub \in \partial \frac{1}{r} \altnorm{f}{B^{s_1}_{p,q_1}}^r$ if and only if $\fsub =\mathcal W^* \mu= \sum_{(j,l)\in I} \mu_{j,l} \phi_{j,l} $ where
	\begin{equation*}
		\mu_{j,l}=\altnorm*{f}{B^{s_1}_{p,q_1}}^{r-q_1} 2^{j s_1 q_1} 2^{jd(\frac{1}{2}-\frac{1}{p}) q_1} \sbra*{\sum\nolimits_{m\in I_j} \absval*{\lambda_{j,m}}^p}^{\frac{1}{p}(q_1-p)} \frac{\lambda_{j,l}}{\absval*{\lambda_{j,l}}^{2-p}}. 
	\end{equation*}
	For $q_3\neq \infty$ (the case $q_3=\infty$ follows along the same lines) we get
	\begin{align*}
		&\, \altnorm*{f^*}{B^{s_3}_{p^\prime,q_3}}^{q_3}/\altnorm*{f}{B^{s_1}_{p,q_1}}^{q_3(r-q_1)}\\
		=&\,  \sum\nolimits_{j\in \Nset_0} 2^{j(s_3+s_1 q_1)q_3} 2^{jd(\frac{1}{2}-\frac{1}{p^\prime}+q_1(\frac{1}{2}-\frac{1}{p}))q_3} \sbra*{\sum\nolimits_{l \in I_j} \absval*{\lambda_{j,l}}^p}^{\frac{q_3}{p}(q_1-p)+\frac{q_3}{p^\prime}}\\
		=&\,  \sum\nolimits_{j\in \Nset_0} 2^{j s_2 q_2} 2^{jd(\frac{1}{2}-\frac{1}{p})q_2} \sbra*{\sum\nolimits_{l\in I_j} \absval*{\lambda_{j,l}}^p}^{\frac{q_2}{p}}\\
		=&\, \altnorm*{f}{B^{s_2}_{p,q_2}}^{q_2},
	\end{align*}
	hence taking the $q_3$-root proves the claim.
	
	For the ``only if'' part note that by duality $\fsub \in \partial \frac{1}{r}\altnorm{f}{\Besovpq}^r$ implies that $f \in \partial \frac{1}{r^\prime}\altnorm{\fsub}{\Besovpqdual[-s]}^{r^\prime}$, hence we also have the implication $f^* \in B^{-s_1}_{p^\prime,q_1^\prime} \cap B^{s_3}_{p^\prime,q_3}$ implies $f\in B^{s_1}_{p,q_1} \cap B^{s_2}_{p,q_2}$.
\end{proof}

The interesting case of the theorem above is if either $s_1<s_2$ or if $s_1=s_2$ and $q_2< q_1$ (and hence $B^{s_1}_{p,q_1} \cap B^{s_2}_{p,q_2}=B^{s_2}_{p,q_2}$ in both cases) since in these cases $ B^{s_3}_{p^\prime,q_3}$ is a proper subspace of $B^{-s_1}_{p^\prime,q_1^\prime}$ and not the other way around. Otherwise -- i.e.\ if $B^{s_1}_{p,q_1} \cap B^{s_2}_{p,q_2}=B^{s_1}_{p,q_1}$ -- the explicit expression of the norm might still be useful. We would like to highlight one special cases of the above theorem. If $q_2=\infty$, that is $f$ is in the largest space with smoothness $s$, then we also obtain $q_3=\infty$. This is interesting because Besov spaces $B^s_{2,\infty}$ are known to be maximal sets for $L^2$-regularization for certain problems (see \cite{HW:17}).

From now on we will always assume sufficient smoothness of the wavelet system in the sense of the previous theorem without further mentioning.

\subsection{Bernstein and Jackson-typeinequalities}
\begin{assumption}\label{ass:projectionBesov}
	Let $(P_j)_{j\in \Nset_0}\colon \mathcal D^\prime(\Omega)\rightarrow \mathcal D^\prime(\Omega)$ be a familiy of linear operators such that for all $p\in(1,\infty)$, $q$, $\tilde q \in [1,\infty]$ and $s$, $t\in \Rset$ with $\max\cbra{\absval s,\absval t}\leq \sigma$ the norm bounds 
	\begin{align*}
		&\text{if }t>s:&\altnorm*{P_j f}{B^t_{p,\tilde q}} &\leq c_1 2^{j (t-s)} \altnorm*{f}{B^s_{p,q}}\\
		&\text{if }t<s:&\altnorm*{(I-P_j) f}{B^t_{p,\tilde q}} &\leq c_2 2^{j (t-s)} \altnorm*{f}{B^s_{p,q}}
	\end{align*}
	with constant $c_1=c_1(p,q,\tilde q,s,t)>0$ and $c_2=c_1(p,q,\tilde q,s,t)>0$  called Bernstein and Jackson inequality, respectively, hold true.
\end{assumption}
\begin{example} The following are possible choices of the operators $(P_j)_{j\in \Nset_0}$.
	\begin{enumerate}
		\item Let $\Omega$ be a Lipschitz domain or the torus as in Section \ref{sec:besovSpacesGeneral} and $\mathcal W$ the wavelet transform defined in \eqref{eq:waveletBesovGeneral:transform}. Set for $j\in \Nset_0$
		\begin{subequations}
		\begin{equation}\label{eq:waveletProjection}
		\begin{aligned}
			P_j f &:= \mathcal W^* Q_j \mathcal W f\\
			\text{where}\qquad \rbra*{Q_j \lambda}_{k^\prime, l^\prime}&:=\begin{cases} \lambda_{k^\prime,l^\prime} & k^\prime\leq j\\ 0 &\text{else}\end{cases}\quad \text{for all } (k^\prime,l^\prime)\in \Nset_0\times I_k.
		\end{aligned}
		\end{equation}
		Due to \eqref{eq:waveletBesovGeneral:norm} this immediately implies a Bernstein and Jackson inequality of the desired form:
		\begin{enumerate}
			\item \emph{Bernstein inequality:} For $t>s$, $\tilde q \in [1,\infty]$ 
			and $f\in \Besovpq[s]$ we get 
			\begin{align*}
				\altnorm*{P_j f}{B^t_{p,\tilde q}}^{\tilde q}&=\sum\nolimits_{k\leq j} 2^{k (t-s) \tilde q} \rbra*{2^{k s} 2^{kd(\frac{1}{2}-\frac{1}{p})} \rbra*{\sum\nolimits_{l\in I_k} \absval{\lambda_{k,l}}^p }^{\frac{1}{p}}}^{\tilde q} \\
				&\leq \sum\nolimits_{k\leq j} 2^{k (t-s) \tilde q} \altnorm*{f}{B^s_{p,\infty}}^{\tilde q} \leq c 2^{j (t-s) \tilde q} \altnorm*{f}{B^s_{p,q}}^{\tilde q}
			\end{align*}
			for some constant $c$ depending on $t$, $s$ and $\tilde q$ only.
			\item \emph{Jackson inequality:} For $t<s$, $\tilde q \in [1,\infty]$ 
			and $f\in \Besovpq[s]$ we obtain
			\begin{align*}
				\altnorm*{(I-P_j) f}{B^t_{p,\tilde q}}^{\tilde q}&=\sum\nolimits_{k > j} 2^{k (t-s) \tilde q} \rbra*{2^{k s} 2^{kd(\frac{1}{2}-\frac{1}{p})} \rbra*{\sum\nolimits_{l\in I_k} \absval{\lambda_{k,l}}^p }^{\frac{1}{p}}}^{\tilde q}\\
				&\leq \sum\nolimits_{k> j} 2^{k (t-s) \tilde q} \altnorm*{f}{B^s_{p,\infty}}^{\tilde q} \leq c 2^{j (t-s) \tilde q} \altnorm*{f}{B^s_{p,q}}^{\tilde q}
			\end{align*}
			where the constant $c$ depends again on $t$, $s$ and $\tilde q$ only.
		\end{enumerate}
		\item Let $\Omega=\Td$ and let the norm of $\Besovpq$ be given by \eqref{eq:fourierBesovNorm}. Then set
		\begin{equation}\label{eq:fourierProjection}
			P_j:=\mathcal F^* \rbra*{\sum\nolimits_{k\leq j}\chi_k(2\pi \cdot)} \mathcal F
		\end{equation}
		\end{subequations}
		for $j\in \Nset_0$ to get the following:
		\begin{enumerate}
			\item \emph{Bernstein inequality:} For $t>s$, $\tilde q \in [1,\infty]$ and
			$f\in \Besovpq[s]$ we get 
			\begin{align*}
				\altnorm*{P_j f}{B^t_{p,\tilde q}}^{\tilde q}&=\sum\nolimits_{k\leq j} 2^{k (t-s) \tilde q} \rbra*{2^{k s \tilde q} \altnorm*{\sum\nolimits_{l\in I_k} \hat f(l) e_l}{L^p(\mathds T^d)}}^{\tilde q} \\
				&\leq \sum\nolimits_{k\leq j} 2^{k (t-s) \tilde q} \altnorm*{f}{B^s_{p,\infty}}^{\tilde q} \leq c 2^{j (t-s) \tilde q} \altnorm*{f}{B^s_{p,q}}^{\tilde q}
			\end{align*}
			for some constant $c$ depending on $t$, $s$ and $\tilde q$ only.
			\item \emph{Jackson inequality:} For $t<s$, $\tilde q \in [1,\infty]$ and
			$f\in \Besovpq[s]$ we obtain
			\begin{align*}
				\altnorm*{(I-P_j) f}{B^t_{p,\tilde q}}^{\tilde q}&=\sum\nolimits_{k > j} 2^{k (t-s) \tilde q} \rbra*{2^{k s \tilde q} \altnorm*{\sum\nolimits_{l\in I_k} \hat f(l) e_l}{L^p(\mathds T^d)}}^{\tilde q}\\
				&\leq \sum\nolimits_{k> j} 2^{k (t-s) \tilde q} \altnorm*{f}{B^s_{p,\infty}}^{\tilde q} \leq c 2^{j (t-s) \tilde q} \altnorm*{f}{B^s_{p,q}}^{\tilde q}
			\end{align*}
			where the constant $c$ depends again on $t$, $s$ and $\tilde q$ only.
		\end{enumerate}
	\end{enumerate}
\end{example}
Further possibilities include projections onto spline or finite element subspaces.

\begin{corollary}\label{cor:projectionSubgradient}
	Let $1<p,q<\infty$, $r=\max\cbra{2,p,q}$ and $s>0$. Let $\ftrue \in B^s_{p,\infty}$ with $\altnorm{\ftrue}{B^s_{p,\infty}}_{\mathcal W}\leq \varrho$ for some $\varrho>0$ and $\fsub \in\partial \frac{1}{r}\altnorm{\ftrue}{B^0_{p,q}}_{\mathcal W}^r$. Then $\altnorm{\fsub}{B_{p^\prime,\infty}^{s(q-1)}}\leq c \varrho^{r-1}$  and if $(P_j)_{j\in \Nset}$ is choosen according to Assumption \ref{ass:projectionBesov} and $a>s(q-1)$ there exists some constant $c>0$ such that
	\begin{equation*}
		\altnorm*{P_j\fsub}{\Besovpqdual[a]} \leq c \varrho^{r-1} 2^{j(a-s(q-1))}\quad\text{and}\quad\altnorm*{(I-P_j)\fsub}{\Besovpqdual[0]} \leq c \varrho^{r-1} 2^{-js(q-1)} .
	\end{equation*}
\end{corollary}
\begin{proof}
	By Theorem \ref{thm:subdiffSmoothness} we get $\fsub \in B_{p^\prime,\infty}^{s(q-1)}$ together with a norm bound. Inserting this bound into Assumption \ref{ass:projectionBesov} where $P_j$ and $I-P_j$ are applied to $\fsub$ gives the desired inequalities.
\end{proof}
We will see in Section \ref{sec:finite_smoothing} that for a wide class of applications these inequalities are enough in order to verify variational source conditions.

\subsection{Deterministic lower bounds}\label{sec:optimal}
In order to see whether the convergence rates implied by variational source conditions are of optimal order we need to find a lower bound on these rates. Such a bound is provided 
by the modulus of continuity, and this lower bound is known to be sharp in 
Hilbert spaces (see \cite{vainikko:87}):
\begin{definition}
	Let $F \colon \Xspace \rightarrow \Yspace$ be continuous and injective, and let $\mathcal K \subset \Xspace$ be   compact. Then the \emph{modulus of continuity} $\omega(\delta, \mathcal K)$ of $(F|_{\mathcal K})^{-1}$ is defined by
	\begin{equation*}
		\omega(\delta, \mathcal K):= \sup \cbra*{\altnorm*{f_1-f_2}{\Xspace} \colon f_1,f_2 \in \mathcal K,
		\altnorm*{F(f_1)-F(f_2)}{\Yspace}\leq \delta }.
	\end{equation*}
\end{definition}
\begin{lemma}[{cf.~\cite[Rem.~3.12]{EHN:96}}]\label{lem:lowerContinuity}
	The worst case error of any (linear or nonlinear) reconstruction method $R\colon \Yspace \rightarrow \Xspace$ on $\mathcal K$ satisfies the lower bound
	\begin{equation}\label{eq:lowerContinuity}
		\sup\cbra*{\altnorm*{f-R(\gobs)}{\Xspace} \colon f\in \mathcal K, \gobs\in\Yspace,
		\altnorm*{F(f)-\gobs}{\Yspace} \leq \delta} \geq \frac{1}{2}\omega(2\delta, \mathcal K).
	\end{equation}
\end{lemma}
\begin{proof}
	Consider $f_1,f_2\in \mathcal K$ such that $\altnorm*{F(f_1)-F(f_2)}{\Yspace}\leq 2\delta$. Then $\gobs:=\frac{1}{2}(F(f_1)+F(f_2))$ satisfies $\altnorm*{F(f_j)-\gobs}{\Yspace}\leq \delta$. Hence, the left hand side $\Delta_R(\delta,\mathcal K)$ of \eqref{eq:lowerContinuity} fulfills
	\begin{align*}
		\Delta_R(\delta,\mathcal K) &\geq \max_{j\in \{1,2\}}\altnorm*{f_j\!-\!R(\gobs)}{\Xspace} \geq \frac{1}{2}\sum\nolimits_{j=1}^2\altnorm*{f_j\!-\!R(\gobs)}{\Xspace} \geq \frac{1}{2}\altnorm*{f_1-f_2}{\Xspace}.
	\end{align*}
	Taking the supremum over all $f_1,f_2$ with the given properties shows \eqref{eq:lowerContinuity}.
\end{proof}
We will prove the following adaptation of \cite[Prop.~4.6]{Daubechies2004}, which estimates the decay of the modulus of continuity if the data has a structure that is compatible with the structure of the Besov space. For this purpose let $(\tilde \phi_{j,l})_{(j,l)\in \tilde I }\subset L^2(\Omega)$ for $\tilde I=\{(j,l):j\in\Nset_0, l\in \tilde I_j\}$ with some countable index sets $(\tilde I_j)_{j\in \Nset_0}$ be an orthonormal system, which might be different from $(\phi_{j,l})_{(j,l)\in  I }$. Further assume that this system defines an equivalent norm on Besov spaces by
\begin{equation}\label{eq:normBesovLower}
\begin{aligned}
	\altnorm*{f}{\Besovpq}&:=\sbra*{\sum\nolimits_{j\in \Nset_0} 2^{jsq}  \altnorm*{\sum\nolimits_{l\in I_j}\tilde\lambda_{j,l}(f)  \tilde \phi_{j,l}}{L^p(\Omega)}^q }^{\frac{1}{q}}\\
	\text{with}\qquad \tilde\lambda_{j,l}(f)&:=\int_{\Omega} f(x) \tilde \phi_{j,l}(x) \,\dd x.
\end{aligned}
\end{equation}
for all $p\in(1,\infty)$, $q\in[1,\infty]$ and all $\absval s \leq \tilde \sigma$.
\begin{proposition}\label{prop:besovLowerBasic}
	Let $1<p<\infty$, $1 \leq q, \tilde q \leq \infty$, $s>0$. In the setting of Lemma \ref{lem:lowerContinuity} set $\Xspace:=\Besovpq[0]$, $\mathcal K :=\cbra{ f\in B^s_{ p,\tilde q} \colon \altnorm{f}{B^s_{ p,\tilde q}}\leq \varrho}$, $\Yspace:=L^2$. Assume that the following holds true:
	\begin{enumerate}
		\item $F(0)=0$.
		\item There exists $(b_{j,l})_{(j,l)\in \tilde I}\subset (0,\infty)$ such that
		\begin{equation*}
			\altnorm{F(f_1)-F(f_2)}{L^2}^2\leq \sum\nolimits_{(j,l)\in \tilde I}  b_{j,l} \absval*{\tilde\lambda_{j,l}(f_1)- \tilde\lambda_{j,l}(f_2)}^2.
		\end{equation*}
		\item\label{prop:besovLowerBasic:sequence} There exists $(f_j)_{j\in \Nset_0}\subset \Xspace$ such that $\tilde \lambda_{k,l}(f_j)=0$ for all $k\neq j$ and for two constants $c_1,c_2>0$ the estimates
		\begin{equation*}
			\frac{1}{c_1}\leq \altnorm*{f_j}{L^2} \leq c_1 \quad \text{and} \quad \frac{1}{c_2}\leq \altnorm*{f_j}{L^p} \leq c_2
		\end{equation*}
		hold true for all $j\in \Nset_0$.
	\end{enumerate}
	Then the modulus of continuity satisfies
	\begin{equation*}
		\omega(\delta,\mathcal K) \geq \frac{1}{c_2} \sup\nolimits_{j\in \Nset_0}  \cbra*{\min \cbra*{\frac{1}{c_1}\rbra*{ \max\nolimits_{l\in \tilde I_j} b_{j,l}}^{-\frac{1}{2}} \delta, \frac{1}{c_2} 2^{-js} \varrho }}.
	\end{equation*}
\end{proposition}
\begin{proof}
	For $j\in \Nset_0$ set
	\begin{equation*}
		w_j:=\min \cbra*{\frac{1}{c_1}\rbra*{ \max\nolimits_{l\in \tilde I_j} b_{j,l}}^{-\frac{1}{2}} \delta, \frac{1}{c_2} 2^{-js} \varrho }.
	\end{equation*}
Straightforward computations show that 
	\begin{align*}
		\altnorm*{F(w_j f_{j})-F(0)}{L^2}\leq \delta, \quad  \altnorm*{w_j f_{j}}{B^s_{p,\tilde q}}\leq \varrho,  \quad\text{and}\quad \altnorm*{w_j f_{j}}{\Besovpq[0]} \geq \frac{w_j}{c_2}.
	\end{align*}
	In particular, $w_j f_{j}\in \mathcal{K}$ for all $j\in \Nset_0$, and obviously also
	$0\in \mathcal K$. This yields 
	\begin{equation*}
		\omega(\delta, \mathcal K)\geq \sup_{j\in \Nset_0} \altnorm*{w_j f_{j}}{\Besovpq[0]} 
		\geq \frac{1}{c_2} \sup_{j\in \Nset_0} \cbra*{\min  \cbra*{\frac{1}{c_1}\rbra*{ \max\nolimits_{l\in \tilde I_j} b_{j,l}}^{-\frac{1}{2}} \delta, \frac{1}{c_2} 2^{-js} \varrho }}.
	\end{equation*}
\end{proof}
Note that the norm defined by the Fourier expansion in \eqref{eq:fourierBesovNorm} is of the form \eqref{eq:normBesovLower} while setting $(\tilde \phi_{j,l})_{(j,l)\in \tilde I }:=( \phi_{j,l})_{(j,l)\in I }$ with $( \phi_{j,l})_{(j,l)\in I }$ as in \eqref{eq:waveletBesovGeneral} leads to an equivalent norm, since there exists a constant $c>0$ independent of $j$ such that 
\begin{equation}\label{eq:waveletLevelNormEquivalence}
	\frac{1}{c} 2^{jd(\frac{1}{2}-\frac{1}{p})}\altnorm*{\lambda_{j,\cdot}}{\ell^p(I_j)}\leq \altnorm*{\sum\nolimits_{l\in I_j} \lambda_{j,l} \phi_{j,l}}{L^p(\Omega)} \leq c 2^{jd(\frac{1}{2}-\frac{1}{p})}\altnorm*{\lambda_{j,\cdot}}{\ell^p(I_j)}
\end{equation}
see \cite[Thm.~3.9.2]{Cohen2003}.
\begin{example}\label{ex:normBesovLower}
	The assumption \ref{prop:besovLowerBasic:sequence} of Proposition \ref{prop:besovLowerBasic} holds true for both types of Besov norms considered in this paper:
	\begin{enumerate}
		\item Assume that there is a constant $c>0$ such that
		\begin{equation*}
			c2^{jd}\leq \absval*{I_j}
		\end{equation*}
		(which is fulfilled for the cases presented in Example \ref{ex:waveletSystem}). Then, choosing $\Gamma_j\subset I_j$ with $\absval{\Gamma_j} \sim 2^{jd}$ we get by \eqref{eq:waveletLevelNormEquivalence} that for $f_j:=\sum_{l\in \Gamma_j} 2^{-j\frac{d}{2}} \phi_{j,l}$ \ref{prop:besovLowerBasic:sequence} holds true for some constants $c_1, c_2 >0$.
		\item\label{ex:normBesovLower:fourier} In the Fourier setting \eqref{eq:fourierBesovNorm} we get $\altnorm{e_z}{L^p}=1$ for all $z\in \Zset^d$ and $p\in [1,\infty]$. Hence we can set  $f_j=e_l$ for some $l\in I_j$ in order to get that \ref{prop:besovLowerBasic:sequence} is fulfilled with $c_1=c_2=1$.
	\end{enumerate}
\end{example}

\section{Finitely smoothing operators}\label{sec:finite_smoothing}
In this section we assume that the forward operator $F_a$ is $a$-times 
smoothing for some $a>0$. More precisely, we assume that 
\begin{subequations}\label{eqs:mappingFa}
\begin{align}
\label{eq:mapping_params}
& p\in(1,2], \; q\in(1,\infty),\; a>\frac{d}{p}-\frac{d}{2}\\
\label{eq:smoothingNorm2}
& \altnorm*{F_a(f_1)-F_a(f_2)}{B^a_{p,q}} \leq L \altnorm*{f_1-f_2}{B^{0}_{p,q}}\\
\label{eq:smoothingNormF}
&\altnorm*{F_a(f_1)-F_a(f_2)}{L^2}\leq L\altnorm*{f_1-f_2}{B^{-a}_{2,2}}\\
\label{eq:smoothingNormFinv}
&\altnorm*{f_1-f_2}{B^{-a}_{2,2}} \leq L\altnorm*{F_a(f_1)-F_a(f_2)}{L^2}
\end{align}
\end{subequations}
for some $L>0$ and all $f_1,f_2\in \Dset\subset B^0_{p,q}$. 

\begin{example}
	Eqs.~\eqref{eqs:mappingFa} are e.g.\ satisfied for the following examples:
	\begin{itemize}
		\item Set $F_a=(I-\Delta)^{-a/2}$ (or more generally, let $F_a$ be a injective elliptic pseudodifferential operators of order $-a$), then $F_a:\Besovpq\to \Besovpq[s+a]$ is bounded and boundedly invertible for all $s\in\Rset$.
		\item In \cite[Lemma 2.9]{HM:19} it was shown that \eqref{eq:smoothingNormF} and \eqref{eq:smoothingNormFinv} follow from the same equations with $F_a$ replaced by $F_a^\prime[\ftrue]$ under a suitable nonlinearity condition. 
		In particular, this covers the following example:\\
		Let $\Omega$ be a bounded Lipschitz domain in $\Rset^d$ for $d\in\cbra{1,2,3}$ and $h_1\in C^\infty(\Omega)$ and $h_2\in C^\infty(\partial\Omega)$ be strictly postive. For $f\in \cbra{f\in L^\infty\colon f(x)\geq 0,\ \forall x\in \Omega, \supp(f)\subset \Omega}$ define $F(f)=u$ where $u$ solves
		\begin{align*}
			\rbra[\big]{-\Delta + f} u&= h_1 &&\text{in }\Omega,\\
			u&=h_2 && \text{on }\partial \Omega
		\end{align*}
		then \eqref{eq:smoothingNormF} and \eqref{eq:smoothingNormFinv} hold true for $a=2$.
	\end{itemize}
\end{example}
Due to  the assumptions \eqref{eq:mapping_params}, \eqref{eq:smoothingNormF} and the continuous 
embedding $\Besovpq[0]\hookrightarrow\Besovpq[-a]$ (see \eqref{eq:besovEmbedSmooth4Int}), 
$F_a\colon B^0_{p,q}(\Td)\to L^2(\Td)$ is well-defined and continuous. 
This allows us to choose 
\begin{equation}\label{eq:RY}
\Xspace=\Besovpq[0](\Td),\qquad 
\Rpen(\cdot)=\frac{1}{r}\altnorm{\cdot}{\Besovpq[0]}_{\mathcal W}^r 
\qquad \mbox{and}\qquad 
\Yspace=B_{2,2}^0(\Td)=L^2(\Td)
\end{equation}
with $r=\max\cbra{2,p,q}$ the modulus of convexity (see \cite{Kazimierski2013}) 
for arbitrary $p\in(1,2]$ and $q\in (1,\infty)$. The penalty term in the Tikhonov functional given by the Besov norm will be expressed via wavelet coefficients as defined in \eqref{eq:waveletBesovGeneral}, hence most constants will depend implicitly on the specific choice of the wavelet system which we will not mention further.

\subsection{Deterministic convergence rates}
We will first derive convergence rates for the deterministic error model \eqref{eq:det_noise}. 
We use the strategy of Theorem \ref{thm:strategy} to obtain a variational source condition first and then apply Proposition \ref{prop:vscRate}\eqref{prop:vscRate:rateGlobal}.

\begin{theorem}\label{thm:finiteVsc}
Assume \eqref{eq:det_noise}, \eqref{eqs:mappingFa} and \eqref{eq:RY} and suppose that 
$\ftrue\in B^s_{p,\infty}$ for some $s\in(0,\frac{a}{q-1})$
with $\altnorm{\ftrue}{B^s_{p,\infty}}\leq\varrho$. Then 
there exists 
a constant $c>0$ such that a variational source condition with
	\begin{equation*}
		\psi(\tau)=c \varrho^\nu \tau^\mu \quad\text{where }\quad \nu
		=\begin{cases}  \frac{qa}{a+s}, & q\geq 2, \\ \frac{2a}{a+s(q-1)},& q\leq 2, \end{cases} 
		\quad \text{and} \quad 
		\mu=\begin{cases}  \frac{q}{2} \frac{s}{a+s}, & q\geq 2,\\ \frac{s(q-1)}{a+s(q-1)}, & q\leq 2\end{cases}
	\end{equation*}
	holds true. Moreover, the Tikhonov functional in  \eqref{eq:tikhonov}  with $F=F_a$ has a minimizer $\fad$, and $\fad$ is unique if $F_a$ is linear. 
If $\alpha$ is chosen by \eqref{eq:vscRate:globalAlpha}, then every minimizer $\fad$ satisfies the error bound 
	\begin{equation}\label{eq:det_rate_finite_smoothing}
		\altnorm*{\fad-\ftrue}{\Besovpq[0]}\leq
\begin{cases} 
c \varrho^\frac{a}{a+s} \delta^\frac{s}{a+s}, & q\geq 2,\\
c \varrho^\frac{a}{a+s(q-1)} \delta^\frac{s(q-1)}{a+s(q-1)}, & q\leq 2
\end{cases}
	\end{equation}
with a constant $c$ independent of $\ftrue$, $\fad$, $\varrho$, and $\delta$. 
\end{theorem}
\begin{proof}
	We apply Theorem \ref{thm:strategy} with the choice $P_j$ as in \eqref{eq:waveletProjection}. Then we see by Corollary \ref{cor:projectionSubgradient} and our assumptions that we can choose $\kappa(j)=c \varrho^{r-1} 2^{-js(q-1)}$.
	
	To verify \eqref{eq:strategyIllposed} denote by $S_a$ the operator  $S_a:= \mathcal W^* \tilde S_a \mathcal W f$ where $(\tilde S_a \lambda)_{j,l}=2^{ja} \lambda_{j,l}$ for all $(j,l)\in I$. Using the relation to Lebesgue spaces \eqref{eq:lesbegueBesov} and \eqref{eq:smoothingNormFinv}, we obtain the estimate
	\begin{equation}\label{eq:finiteSigma}
	\begin{aligned}
		\pairing*{P_j \fsub}{\ftrue-f}&=\pairing*{S_a P_j \fsub}{ S_{-a}(\ftrue-f)} \\
		&\leq c \altnorm*{S_a P_j \fsub}{L^{p^\prime}} \altnorm*{ S_{-a}(\ftrue-f)}{L^2} \altnorm*{1}{L^{\frac{2p}{2-p}}(\Td)}\\
		&\leq c \altnorm*{S_a P_j \fsub}{B^0_{p^\prime,2}} \altnorm*{\ftrue-f}{B^{-a}_{2,2}}\\
		&\leq cL \altnorm*{P_j \fsub}{B^a_{p^\prime,2}} \altnorm*{F_a(\ftrue)- F_a(f)}{L^2}.
	\end{aligned}
	\end{equation}

	By Corollary \ref{cor:projectionSubgradient} we can hence choose $\gamma=0$ and
	\begin{equation*}
		\sigma(j)=c \varrho^{r-1} 2^{j\rbra*{a -s(q-1) }}
	\end{equation*}
	in \eqref{eq:strategyIllposed} with $c$ depending $c>0$ depending on the wavelet system and the parameters $s,p,q,a$.
	
	Now Theorem \ref{thm:strategy} implies that a variational source condition holds true with
	\begin{equation*}
		\psi_\text{vsc}(\tau)=\inf_{j\in \Nset_0} c \sbra*{ \varrho^{r-1} 2^{j\rbra*{a -s(q-1) }} \sqrt{\tau} +\varrho^{r} 2^{-js(q-1)r^\prime} }.
	\end{equation*}
	Choosing $j$ such that $2^j \sim (\varrho/\sqrt{t})^\tau$ with $\tau= \frac{1}{s(q-1)(r^\prime-1)+a}$ and we can estimate
	\begin{equation*}
		\psi_\text{vsc}(\tau)\leq c \varrho^{r- \frac{s(q-1)r^\prime}{a+s(q-1)(r^\prime-1)}} \tau^{\frac{1}{2} \frac{s(q-1)r^\prime}{s(q-1)(r^\prime-1)+a}}.
	\end{equation*}
	Now use that for $q\leq 2$ we have $r=r^\prime=2$ and for $q\geq 2$ we have $r=q$ and $r^\prime=q^\prime$.

The existence of $\fad$ follows from standard results (see, e.g., \cite[Thm. 3.22]{Scherzer_etal:09}) using the compactness of the embedding $\Besovpq[0]\hookrightarrow B^{-a}_{2,2}$ (see \eqref{eq:besovEmbedSmooth4Int})  
and \eqref{eq:mapping_params}) 
and \eqref{eq:smoothingNormF}. Uniqueness of $\fad$ for linear operators 
is obvious by strict convexity. 
From Proposition \ref{prop:vscRate} we obtain $\Delta_{\Rpen}(\fad,\ftrue) \leq \psi_\text{vsc}(\delta^2)$, and 
via Example \ref{ex:bregmanNormBound}\eqref{ex:bregmanNormBound:convex} (with $r=\max(2,q)$ 
as discussed after \eqref{eq:RY}) this yields the convergence rate \eqref{eq:det_rate_finite_smoothing}. 
\end{proof}

In practice the parameters $s$ and $\varrho$ describing the smoothness of $\ftrue$ 
are usually unknown, of course, and hence the a-priori rule 
\eqref{eq:vscRate:globalAlpha} is not implementable. Therefore, a-posteriori 
rules such as the discrepancy principle are used, under which the same 
error bounds can be shown without prior knowledge of $s$ and $\varrho$
(see, e.g., \cite{HM:12}). 

\subsection{Extensions}
In this subsection we discuss extensions of 
the results of the previous subsection resulting from different penalty terms and data-fidelity terms respectively.

\begin{theorem}\label{thm:differentSmooth}
Let the Assumptions of Theorem \ref{thm:finiteVsc} hold true, but in \eqref{eq:RY} set 
$\Xspace = \Besovpq[\tilde s]$ and 
$\Rpen(\cdot)=\frac{1}{r}\altnorm{\cdot}{\Besovpq[\tilde s]}_{\mathcal W}^r$ for $\tilde s\in \Rset$. Further replace the last inequality 
in \eqref{eq:mapping_params} by $a^*:=a+\tilde s>\frac{d}{p}-\frac{d}{2}$, and replace 
\eqref{eq:smoothingNorm2} by $\altnorm*{F_a(f_1)-F_a(f_2)}{B^{a^*}_{p,q}} \leq L \altnorm*{f_1-f_2}{B^{\tilde{s}}_{p,q}}$. 
Let $\ftrue\in B^s_{p,\infty}$ for some $s\in \Rset$ such that $s^*:=s-\tilde s \in(0,\frac{a^*}{q-1})$.  Then the Tikhonov minimizer $\fad$ in \eqref{eq:tikhonov} exists, and for $\alpha$ chosen by \eqref{eq:vscRate:globalAlpha} and $q\geq 2$ 
it satisfies
	\begin{equation*}
			\altnorm*{\fad-\ftrue}{\Besovpq[\tilde s]}
\leq 
c \varrho^\frac{a^*}{a^*+s^*} \delta^\frac{s^*}{a^*+s^*}.
	\end{equation*}
\end{theorem}
\begin{proof}
	The proof is analogous to the proof of Theorem \ref{thm:finiteVsc}. Here we obtain 
	$\kappa(j) = c\varrho^{r-1}2^{-js^*(q-1)}$ and $\sigma(j) = c\varrho^{r-1}2^{j(a^*-s^*(q-1))}$.  
\end{proof}

\begin{remark}
Suppose the constraint $\ftrue\in \Dset$ is incorporated in the penalty term 
$\Rpen$ by replacing it 
by $\widetilde{\Rpen}(f):=\Rpen(f)+ \chi_{\Dset}(f)$ where $\chi_{\Dset}(f):=0$ if $f\in\Dset$ 
and $\chi_{\Dset}(f):=\infty$ else. Then $\partial\widetilde{\Rpen}(\ftrue)= \partial \Rpen(\ftrue) 
+ \partial \chi(\ftrue)$ by the sum rule. $\partial \chi(\ftrue)$ coincides with the normal cone 
at $\ftrue$ and differs from $\{0\}$ if $\ftrue$ belongs to the boundary of $\Dset$. 
In this case $\partial\widetilde{\Rpen}(\ftrue)$ may contain elements of higher smoothness than 
$\partial\Rpen(\ftrue)$ leading to faster rates of convergence
(see \cite{FH:11} and \cite[\S 5.4]{EHN:96}). 
\end{remark}

\begin{theorem}\label{thm:otherDataTerm}
The error bound \eqref{eq:det_rate_finite_smoothing} in 
Theorem \ref{thm:finiteVsc} remains true if we replace $\Yspace=L^2$ by $\Yspace=\Xspace=B^0_{p,q}$ 
for $p\in(1,2]$ and $q\in(1,\infty)$, if we replace the Tikhonov function \eqref{eq:tikhonov} by 
\eqref{eq:tikhonov_det} with arbitrary $t\geq 1$ and if we replace Assumption \eqref{eqs:mappingFa} by 
\[
\tfrac{1}{L}\altnorm*{f_1-f_2}{B^{-a}_{p,q}}\leq \altnorm*{F_a(f_1)-F_a(f_2)}{B^0_{p,q}} 
\leq L\altnorm*{f_1-f_2}{B^{-a}_{p,q}}.
\] 
\end{theorem}
\begin{proof}
In the proof of Theorem \ref{thm:finiteVsc} we have to adapt the estimate \eqref{eq:finiteSigma} in the 
following way:
\begin{align*}
	\pairing*{P_j \fsub}{\ftrue-f}\leq c \altnorm*{P_j \fsub}{B^a_{p^\prime,q^\prime}} \altnorm*{\ftrue-f}{B^{-a}_{p,q}}.
\end{align*}
The norm $\altnorm{P_j \fsub}{B^a_{p^\prime,q^\prime}}$ can be estimated as before. The exponents $\mu$ in the 
source condition are $\mu=\frac{q}{t}\frac{s}{a+s}$ for $q\geq 2$ and $\mu=\frac{2}{t}\frac{s(q-1)}{a+s(q-1)}$ 
for $q\in (1,2]$ in this case.  
\end{proof}


\begin{corollary}\label{coro:log}
Under the assumptions of Theorem \ref{thm:finiteVsc} with $q=2$ choose 
$\js:=\lfloor-\frac{1}{2a}\ln_2\alpha\rfloor$. Then
\begin{align}\label{eq:logLp}
\altnorm*{P_{\js}\fad-\ftrue}{L^p}\leq c\altnorm*{P_{\js}\fad-\ftrue}{B^0_{p,p}}
\leq c \varrho^{\frac{a}{s+a}}\delta^{\frac{s}{s+a}}\left(\ln \delta^{-1}\right)^{\frac{1}{p}-\frac{1}{2}}.
\end{align}
\end{corollary}

\begin{proof}
Setting $h=\fad-\ftrue$ we get from H\"older's inequality that 
\[
\altnorm*{P_{\js}h}{B^0_{p,p}}
= \left(\sum\nolimits_{j=0}^{j_*}1\cdot \altnorm*{h_j}{L^p}^p\right)^{1/p}
\leq \js^{1/p-1/2}\altnorm*{P_{\js}h}{B^0_{p,2}}. 
\]
As $\alpha \sim (\delta/\varrho)^{2a/(s+a)}$ by \eqref{eq:vscRate:globalAlpha} we have 
$\altnorm*{(I-P_{\js})\ftrue}{B^0_{p,p}}\leq c2^{-\js s}\altnorm*{\ftrue}{B^s_{p,\infty}}
\leq c\varrho \sqrt{\alpha}^{s/a} = c\varrho^{a/(s+a)} \delta^{s/(s+a)}$. 
Combining both inequalities yields the assertion. 
\end{proof}

\begin{remark}
In the setting of Theorem \ref{thm:otherDataTerm} the 
projection $P_{\js}$ in \eqref{eq:logLp} can be omitted. 
This follows after some computations comparing the value of the Tikhonov functional for 
$\fad$ and $f^{*}_\alpha:= P_{\js}\fad + (I-P_{\js})\ftrue$. 
\end{remark}

\begin{remark}\label{rem:strategyLipschitz}
	The basic idea of Theorem \ref{thm:finiteVsc} can be generalized as follows: Assume that there exist Banach spaces $\Xspace_{\mathrm{Lip}}$ and $\Xspace_{\mathrm s}$ such that the embeddings $\Xspace\hookrightarrow \Xspace_{\mathrm s} \hookrightarrow \Xspace_{\mathrm{Lip}}$ are continuous. Let $(\Xspace_j^\prime)_{j\in \Nset}$ be a sequence of finite dimensional subspaces such that
	\begin{equation*}
		\overline{\bigcup_{j\in \Nset} \Xspace_j^\prime}^{\altnorm{\cdot}{\Zspace}}=\Zspace\qquad  \text{for }\Zspace\in \cbra{\Xspace^\prime, \Xspace_{\mathrm s}^\prime,\Xspace_{\mathrm{Lip}}^\prime},
	\end{equation*}
	and let $P_j$ be a projection onto $\Xspace^\prime_j$. Assume that $\ftrue$ is such that $\fsub\in \Xspace_{\mathrm s}^\prime$ and that for all elements $h\in \Xspace_{\mathrm s}^\prime$ the generalized Bernstein and Jackson inequalities
	\begin{align*}
		 \altnorm*{P_j h}{\Xspace_{\mathrm{Lip}}^\prime} &\leq \tilde \sigma(j) \altnorm*{h}{\Xspace_{\mathrm s}^\prime}\\
		 \text{and}\qquad \altnorm*{(I-P_j) h}{\Xspace^\prime} &\leq \tilde \kappa(j)\altnorm*{h}{\Xspace_{\mathrm s}^\prime}
	\end{align*}
	hold true. If the operator $F$ fulfills the Lipschitz estimate
	\begin{equation*}
		\altnorm*{f_1-f_2}{\Xspace_{\mathrm{Lip}}}\leq L \altnorm*{F(f_1)-F(f_2)}{\Yspace}\qquad \text{for all } f_1,f_2\in \Xspace
	\end{equation*}
	then \eqref{eq:strategy} is fulfilled with $\kappa(j)= \tilde \kappa(j) \altnorm*{\fsub}{\Xspace_{\mathrm s}^\prime}$, $\sigma(j)= L \tilde \sigma(j)\altnorm*{\fsub}{\Xspace_{\mathrm s}^\prime}$ and $\gamma=0$.
\end{remark}

\subsection{Statistical convergence rates}
As a variational source condition is fulfilled and the regularization functional fulfills Assumption \ref{ass:bregmanNormBound} we only need to show that the operator also fulfills Assumption \ref{ass:interpolation}. We then obtain convergence rates via Theorems \ref{thm:whiteNoiseRate} and \ref{thm:finiteVsc}.

\begin{lemma}\label{lem:interpolationFinite}
Suppose that $a>d/2$ and \eqref{eq:smoothingNorm2} hold true. Then the operator $F_a$ fulfills Assumption \ref{ass:interpolation} 
	and \eqref{eq:interpolation_norms}
	with $\beta=1-\frac{d}{2a}$ and $\gamma=\frac{d}{2ar}$. Moreover, there exists $c>0$ such that
	\begin{align}\label{eq:interpolationFa}
		\altnorm*{g}{B^{d/2}_{p,1}}&\le c \altnorm*{g}{L^2}^{1-\frac{d}{2a}} \altnorm*{g}{B^{a}_{p,q}}^{\frac{d}{2a}}
			\qquad\mbox{for all }g\in B^{a}_{p,q}.
			\end{align}
\end{lemma}
\begin{proof}
	By K-interpolation theory (see \cite[Sec.~2.4.2]{Triebel2010}) there exists a constant $c>0$ such that
	\begin{align*}
		\altnorm*{g}{B^{d/2}_{p,1}}&\le c \altnorm*{g}{B^{0}_{p,2}}^{1-\frac{d}{2a}} \altnorm*{g}{B^{a}_{p,q}}^{\frac{d}{2a}},
	\end{align*}
	and since $\altnorm{\cdot}{B^0_{p,2}(\Td)}\leq \altnorm{\cdot}{L^2(\Td)}$ for $p\leq 2$, 
	this implies \eqref{eq:interpolationFa}. Using \eqref{eq:smoothingNorm2} and 
Assumption \ref{ass:bregmanNormBound}	we obtain 
	\begin{align*}
		\altnorm*{F_a (f_1)-F_a (f_2)}{B^{a}_{p,q}}\le L\altnorm*{f_1-f_2}{B^0_{p,q}}\le L \rbra*{C_\Delta^{-1}\Delta_\Rpen(f_1,f_2)}^\frac{1}{r}.
	\end{align*}
This together with \eqref{eq:interpolationFa} for $g= F_a(f_1)-F_a(f_2)$ yields Assumption \ref{ass:interpolation} 
and \eqref{eq:interpolation_norms}. 
\end{proof}
Lemma \ref{lem:interpolationFinite} can not only be used to derive convergence rates, but also existence 
of a minimizer: 

\begin{proposition}\label{prop:existence}
Suppose that $a>d/2$ and $F_a$ satisfies  \eqref{eq:mapping_params}--\eqref{eq:smoothingNormF}. 
Then for the noise model \eqref{eq:rand_noise} with $\rand$ satisfying 
Assumption \ref{ass:dev_ineq}, the Tikhonov functional 
in \eqref{eq:tikhonov} with $F=F_a$ has a global minimizer $\fad$ almost surely.
\end{proposition}

\begin{proof}
Note that the data fidelity term in \eqref{eq:tikhonov} is not bounded 
from below in general, and therefore standard results in the literature such as \cite[Thm.~3.22]{Scherzer_etal:09}  do not apply. However, with the help of Lemma \ref{lem:interpolationFinite} we can show  
coercivity of the entire Tikhonov functional in $B^0_{p,q}$ if $N:=\altnorm{\gobs}{B^{-d/2}_{p',\infty}}<\infty$ 
(which is true with probability $1$ by Assumption~\ref{ass:dev_ineq}). To this end we bound the mixed term 
as follows using \eqref{eq:interpolationFa} and H\"older's inequality $xy\leq cx^{4a/(2a+d)} + \frac{1}{2}y^{4a/(2a-d)}$: 
\begin{align*}
\pairing*{\gobs}{F(f)} &\leq N\altnorm*{F(f)}{B^{d/2}_{p,1}} 
\leq CN\altnorm*{F(f)}{B^a_{p,q}}^{d/2a} \altnorm*{F(f)}{L^2}^{1-d/2a}\\
&\leq c N^{\frac{4a}{2a+d}}\big(\altnorm*{F(f)-F(0)}{B^a_{p,q}} + \altnorm*{F(0)}{B^a_{p,q}}\big)^{2\frac{d}{2a+d}}
+\tfrac{1}{2} \altnorm*{F(f)}{L^2}^2\\
&\leq c N^{\frac{4a}{2a+d}}\altnorm*{f}{B^0_{p,q}}^{2\frac{d}{2a+d}}+A
+\tfrac{1}{2} \altnorm*{F(f)}{L^2}^2
\end{align*}
with $A:=cN^{4a/(2a+d)}\altnorm{F(0)}{B^{d/2}_{p,1}}^{2d/(2a+d)}$ and a generic constant $c$. 
Plugging this into the Tikhonov functional yields 
\begin{align*}
&\tfrac{1}{2}\altnorm*{F(f)}{L^2}^2-\pairing*{\gobs}{F(f)} + \alpha\altnorm*{f}{B^0_{p,q}}^r\\
\geq& -cN^{\frac{4a}{2a+d}}\altnorm*{f}{B^0_{p,q}}^{2\frac{d}{2a+d}}+\alpha\altnorm*{f}{B^{0}_{p,q}}^r+A,
\end{align*} 
and as $r\geq 2$ the right hand side tends to $\infty$ as $\altnorm{f}{B^0_{p,q}}\to \infty$. This shows that 
a minimizing sequence $(f_n)$ of the Tikhonov functional must be bounded in $B^0_{p,q}$. 
As $B^0_{p,q}$ is reflexive, by the Banach-Alaoglu theorem there exists a subsequence $f_{n_k}$ and 
$f\in B^0_{p,q}$ such that $f_{n_k}\rightharpoonup f$ for $k\to \infty$. Since the embedding 
$B^0_{p,q}\hookrightarrow B^{-a}_{2,2}$ is 
compact, we have $\lim_{k\to\infty}\altnorm{f_{n_k}-f}{B^{-a}_{2,2}}=0$, and by \eqref{eq:smoothingNormF} also 
$\lim_{k\to\infty}\altnorm{F(f_{n_k})-F(f)}{L^2}=0$. Now it follows from \eqref{eq:interpolation_norms} 
and the boundedness of $\altnorm{f_{n_k}}{B^0_{p,q}}$ that 
$\altnorm{F(f_n)-F(f)}{B^{d/2}_{p,1}}$ tends to $0$ as $k\to\infty$. Together with the weak lower semicontinuity 
of $\altnorm{\cdot}{B^0_{p,q}}^r$ it follows that $f$ minimizes the Tikhonov functional. 
\end{proof}

Together with Theorem \ref{thm:finiteVsc}  we find the following:
\begin{theorem} \label{thm:prob_finiteRate}
Assume \eqref{eq:rand_noise} with $\rand$ satisfying Assumption 
\ref{ass:dev_ineq}, \eqref{eqs:mappingFa}, \eqref{eq:RY} with $a>d/2$ and $q\geq 2$. 
Moreover, suppose that $\ftrue\in B^s_{p,\infty}$ for some $s\in(0,\frac{a}{q-1})$
with $\altnorm{\ftrue}{B^s_{p,\infty}}\leq\varrho$. Then $\fad$ (as in 
Proposition \ref{prop:existence}) for an optimal choice of $\alpha$ as specified in the proof
satisfies the error bound
	\begin{align*}
&\prob\left[\altnorm*{\fad-\ftrue}{\Besovpq[0]}> \left(c+t\right) \varrho^\frac{a+d/2}{a+s+d/2} \varepsilon^\frac{s}{a+s+d/2} 
		 \right] \le \exp\rbra*{-C_\rand t^{\mu\rbra*{\frac{q}{2}+\frac{(q-2)d}{4a}}}}&
	\end{align*}	
for all $t>0$.
In particular, for any $\sigma\geq 1$ we have
	\begin{align*}
		\E\left(\altnorm*{\fad-\ftrue}{\Besovpq[0]}^\sigma\right)^{1/\sigma}\leq
C \varrho^\frac{a+d/2}{a+s+d/2} \varepsilon^\frac{s}{a+s+d/2}. 
	\end{align*}
Here $c$ and $C$ are positive constants independent of $\ftrue$, $\fad$, 
$\varepsilon$, and $\varrho$. 
\end{theorem}

\begin{proof}
	As $q\ge 2$, by Theorem \ref{thm:finiteVsc} a variational source condition with $\psi(\tau)=c\varrho^{\frac{2a}{a+s}}\tau^\frac{s}{a+s}$ holds true and hence we can use Example \ref{ex:conjugate}, \eqref{eq:conjugate:hoelder} with the calculus rules for Fenchel duals to obtain that
	\begin{equation*}
		\varphi_\psi (\tau)= C \varrho^\frac{2qa}{2(a+s)-qs} \tau^\frac{qs}{2(a+s)-qs}.
	\end{equation*}
	Thus Theorem \ref{thm:whiteNoiseRate}\eqref{thm:whiteNoiseRate:rate} together with Lemma \ref{lem:interpolationFinite}  gives
	\begin{align*}
		\Delta_\mathcal{R}(\fad,\ftrue) \le C \alpha^{-\frac{q+qd/(2a)}{q+d(q-2)/(2a)}} \altnorm*{\varepsilon\rand}{B^{-d/2}_{p^\prime,\infty}}^{\frac{2q}{q+d(q-2)/(2a)}}+C \varrho^\frac{2qa}{2(a+s)-qs}\alpha^\frac{qs}{2(a+s)-qs}.
	\end{align*}
	By the choice $\alpha \sim \varrho^\frac{qa+d\frac{q}{2}(1-\frac{2}{q}))}{a+s+d/2} \varepsilon^\frac{2(a+s)-qs}{a+s+d/2}$ and using \eqref{eq:bregmanNormBound:convex} we find
	\begin{align*}
		\altnorm*{\fad-\ftrue}{\Besovpq[0]}\le C 
		\varrho^\frac{a+d/2}{a+s+d/2} \varepsilon^\frac{s}{a+s+d/2} \left(1+\altnorm*{\rand}{B^{-d/2}_{p^\prime,\infty}}^{\frac{2}{q+d(q-2)/(2a)}}\right),
	\end{align*}
	so the deviation inequality given by Assumption \ref{ass:dev_ineq} completes the proof.
\end{proof}

\subsection{Lower bounds}
\begin{theorem}\label{thm:optFinite}
Suppose that $F_a$  satisfies \eqref{eq:smoothingNormF}. Then 
there cannot exist a reconstruction method $R\colon L^2 \rightarrow \Besovpq[0]$ for the operators $F_a$ satisfying the worst case error bound
	\begin{equation*}
		\sup \cbra*{\altnorm*{f-R(\gobs)}{\Besovpq[0]} \colon \altnorm{f}{B^s_{p,\infty}}\leq \varrho,\altnorm*{F_a (f)-\gobs}{L^2} \leq \delta} = o\rbra*{\varrho^\frac{a}{s+a} \delta^\frac{s}{s+a}}
	\end{equation*}
	Hence for $q\geq 2$ the rate in Theorem~\ref{thm:finiteVsc} is optimal up to the value of the constant.
\end{theorem}
\begin{proof}
	The assumptions of Proposition \ref{prop:besovLowerBasic} are fulfilled with $b_{j,l}=c 2^{-2ja}$ for some $c>0$ by \eqref{eq:smoothingNormF}. Hence we obtain
	\begin{equation*}
		\omega(\delta, \mathcal K)\geq c \max_{k\in \Nset_0} \cbra*{\min \cbra*{2^{-ks} \varrho, 2^{ka} \delta}}
		\geq c \varrho^\frac{a}{s+a} \delta^\frac{s}{s+a}
	\end{equation*}
where we have chosen $k\in \Nset_0$ such that the terms are balanced, i.e.~$2^k\sim (\frac{\varrho}{\delta})^{1/(s+a)}$. 
Now the claim follows by Lemma \ref{lem:lowerContinuity}.
\end{proof}

\begin{remark}
	The statement of Theorem \ref{thm:optFinite} remains valid in the setting of Theorem \ref{thm:differentSmooth} if one replaces $a$ and $s$ by $a^*$ and $s^*$ respectively.
\end{remark}

Lower bounds for the statistical convergence rates can be concluded from results in \cite{Donoho1995}. 
Instead of the continuous Gaussian white noise model they consider  an $n$-dimensional \emph{normal means model}.
However as their results in  \cite[Thms.~7 and 9]{Donoho1995} do not depend on the dimension $n$ one can send $n$ to infinity, so that the Le Cam distance of the two models goes to zero (compare \cite[\S 1]{Nickl2016}) and thus conclude for general estimators $S=S(\gtrue+\varepsilon W)\in B^{s^*}_{p,q}$:
\begin{corollary}
	 We have
	\begin{align*}
		\inf\nolimits_{S}\sup\nolimits_{\altnorm{\gtrue}{B^{s^{**}}_{p,\infty}}\leq \varrho}\E\left(\altnorm*{\gtrue-S(\gtrue+\varepsilon W)}{B^{s^*}_{p,q}}\right)\ge c \varrho^{\frac{s^*+d/2}{s^{**}+d/2}}\varepsilon^\frac{s^{**}-s^*}{s+d/2},
	\end{align*}
	with $c$ depending on $s^*,s^{**},p,q$.
\end{corollary}
Assume additionally on $F_a$ that $F_a:\Besovpq\to \Besovpq[s+a]$ is surjective and 	 $\altnorm{f}{B^s_{p,\infty}}\leq L\altnorm{F_a f}{B^{s+a}_{p,\infty}}$  for all $f\in B^s_{p,\infty}$.
By setting $s^*=a$ and $s^{**}=s+a$ we find for $F_a$  that
	\begin{align*}
		\inf\nolimits_{S}\sup\nolimits_{\altnorm{F_a(\ftrue)}{B^{s+a}_{p,\infty}}\leq \varrho}\E\left(\altnorm*{F_a(\ftrue)-S(F_a(\ftrue)+\varepsilon W)}{B^a_{p,q}}\right)\ge C \varrho^{\frac{a+d/2}{s+a+d/2}}\varepsilon^\frac{s}{s+a+d/2}.
	\end{align*}
	 	Now by \eqref{eq:smoothingNormF} we have for reconstruction methods $R$ that
	\begin{align*}
		\altnorm*{\ftrue-R(F_a(\ftrue)+\varepsilon W)}{\Besovpq[0]}\ge \frac{1}{L} \altnorm*{F_a(\ftrue)-F_a R(F_a(\ftrue)+\varepsilon W)}{B^{a}_{p,q}}.
	\end{align*}
Thus we get a
lower bound coinciding with the upper bound in Theorem \ref{thm:prob_finiteRate} for $q\geq 2$: 
	\begin{align*}
		\inf\nolimits_{R}\sup\nolimits_{\altnorm{\ftrue}{B^s_{p,\infty}}\leq L\varrho}\E\left(\altnorm*{\ftrue-R(F_a(\ftrue)+\varepsilon W)}{B^0_{p,q}}\right)\ge C \varrho^{\frac{a+d/2}{s+a+d/2}}\varepsilon^\frac{s}{s+a+d/2}.
	\end{align*}

\subsection{Numerical validation}\label{sec:numeric}
We are considering a problem of the type \eqref{eqs:mappingFa} where $F_a:\Besovpq[0](\mathds T)\rightarrow L^2(\mathds T)$ is given by $F_a:=(I-\partial_x^2)^{-1}$, that is we have $a=2$, with a deterministic error model. The true solution $\ftrue$ is given by a continuous, piecewise linear function, therefore $\ftrue \in B^{s}_{p,\infty}$ for $s\leq 1+1/p$. 
As for $q= 2$ the obtained convergence rates are of optimal order, we test for different values of $p$ if they are also achieved numerically using the sequential discrepancy principle on the grid $\alpha_j=2^{-j}$ with parameter $\tau=2$, see \cite{Anzengruber2014} for details.

Numerical computations are carried out in \textsc{matlab}. To obtain an efficient implementation of the operator $F_a$ we use the FFT on a grid with $2^{10}$ nodes. For the Besov norm we use the wavelet decomposition of the Wavelet toolbox with periodic db7-wavelets. An inverse crime is avoided via generating data on a finer grid and undersampling. In order to obtain the minimizer of the Tikhonov functional we use the extension of the Chambolle-Pock algorithm to Banach spaces with a constant parameter choice rule, see \cite[Thm.~6]{Hohage2014a}, where the iterations are stopped when the current step gets small compared to the first. Note that the steps of this algorithm become especially simple since the considered spaces are $2$-convex. The duality mappings are evaluated with the help of Theorem \ref{thm:subdiffSmoothness}. For further details see Appendix \ref{sec:detailsSimulation}.

\begin{figure}[ht]
	\includegraphics[width=\textwidth]{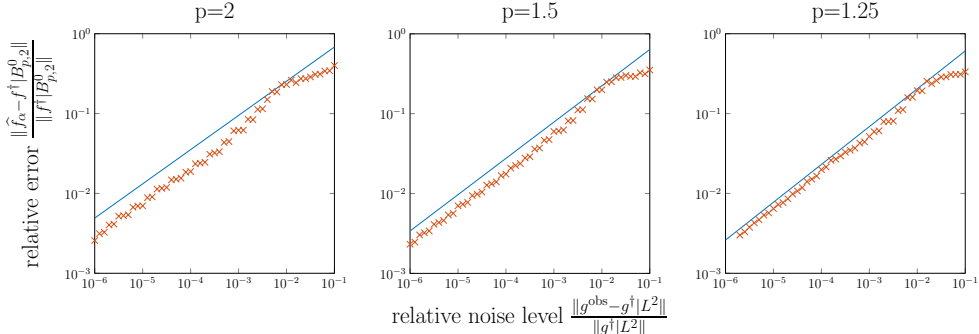}\vspace{-1.5em}
	\caption{Convergence rates in $B^0_{p,2}$ for different values of $p$. 
The crosses indicate reconstruction errors and the lines convergence 
rates predicted by Theorem \ref{thm:finiteVsc}.}\label{fig:rate}
\end{figure}
We tested which convergence rate we observe if we choose $\Rpen(f)=\frac{1}{2} \altnorm{f}{B^0_{p,2}}^2$ for different values of $p$.
The results of this test are shown in Figure \ref{fig:rate}. It can be seen that for the tested values of $p$ the observed rates coincide quite well with the predicted optimal rates. The staircase behavior of the reconstruction error plots -- best visible for $p=2$ -- is due to the sequential discrepancy principle; if for different noise levels the relative error is almost constant, then the same regularization parameter $\alpha$ was chosen. 
For $p=1.25$ the last three points are omitted in the plot since our code did not produce 
solutions satisfying the discrepancy principle.

\section{Backwards heat equation}\label{sec:BH}
As a second problem we consider the backwards heat equation on $\Td$, that is $(T_\BH f) (x)= u(x,\bar t)$ for some fixed $\bar t >0$ where $u$ solves 
\begin{align*}
	\partial_t u &= \Delta u &&\text{in }\Td \times (0,\overline{t})\\
	u(\cdot,0) &= f &&\text{on }\Td.
\end{align*}
Note that $T_\BH$ can be conveniently expressed via the Fourier series transform $\mathcal F$ as
\begin{equation*}
	T_\BH f = \mathcal F^* \exp\rbra*{-\bar t \absval{\cdot}^2} \mathcal F f.
\end{equation*}

\subsection{Deterministic convergence rates}
\begin{theorem}\label{thm:vscBH}
Let $p\in(1,2]$, $q\in (1,\infty)$ and $s>0$, and suppose that 
\eqref{eq:det_noise} and \eqref{eq:RY} hold true for  
$\ftrue\in B^s_{p,\infty}$ with $\altnorm{\ftrue}{B^s_{p,\infty}}\leq \varrho$. 
Then $\ftrue$ satisfies a variational source condition with
	\begin{equation*}
		\psi(\tau)=c \varrho^r  \sbra*{\frac{\sqrt \tau}{\varrho} \rbra*{3+\frac{\varrho}{\sqrt \tau}}^{1/2} +\rbra*{\ln \rbra*{\rbra*{3+\frac{\varrho}{\sqrt \tau}}^{1/2}}}^{-s(q-1)r^\prime/2}}.
	\end{equation*}
Moreover, $\fad$ in \eqref{eq:tikhonov} with $F=T_\BH$ 
exists and is unique for any $\alpha>0$. For the parameter choice rule 
\eqref{eq:vscRate:globalAlpha} and $q\geq 2$ it satisfies the error bound
	\begin{equation}\label{eq:heat_rate}
		\altnorm*{\fad-\ftrue}{\Besovpq[0]}\leq
c \varrho \rbra*{\ln \rbra*{\frac{\varrho}{\delta}}}^{-s/2} 
		\qquad \text{as } \delta\rightarrow 0
	\end{equation}
	with a constant $c>0$ independent of $\ftrue$, $\delta$, and $\varrho$. 
\end{theorem}
\begin{proof}
	As in the proof of Theorem \ref{thm:finiteVsc}, we apply Theorem \ref{thm:strategy} but this time with the choice $P_j$ as in \eqref{eq:fourierProjection}. Again we obtain by Corollary \ref{cor:projectionSubgradient} and our assumptions that we can choose $\kappa(j)=c \varrho^{r-1} 2^{-js(q-1)}$.
	
	In order to verify \eqref{eq:strategyIllposed} note that
	\begin{align*}
		\altnorm*{P_j^* (\ftrue-f)}{L^2}^2 &=\sum_{k=0}^j \sum_{l\in I_k} \absval*{\widehat{(\ftrue-f)}(l)}^2
		\leq \rbra*{\ee^{2^{2j} \bar t }}^2 \sum_{k=0}^j \sum_{l\in I_k}  \absval*{\ee^{-\absval{l}^2 \bar t} \widehat{(\ftrue-f)}(l)}^2\\
		&\leq \rbra*{\ee^{2^{2j} \bar t}}^2 \altnorm*{T(\ftrue-f)}{L^2}^2.
	\end{align*}
	Therefore, we can estimate
	\begin{align*}
		\pairing*{P_j \fsub}{\ftrue-f} & \leq \altnorm*{\fsub}{L^{p^\prime}} \altnorm*{1}{L^{\frac{2p}{2-p}}}\altnorm*{P_j^* (\ftrue-f)}{L^2}\\
		 &\leq c \altnorm*{\fsub}{B_{p^\prime,\infty}^{s(q-1)}} \ee^{2^{2j} \bar t} \altnorm*{T(\ftrue-f)}{L^2} \\
		 &\leq c \varrho^{r-1} \ee^{2^{2j} \bar t} \altnorm*{T(\ftrue-f)}{L^2}
	\end{align*}
	and hence choose $\sigma(j)=c \varrho^{r-1} \ee^{2^{2j} \bar t}$ and $\gamma=0$.
	
	This implies by Theorem \ref{thm:strategy} that a variational source condition with
	\begin{equation*}
		\psi_\text{vsc}(\tau)=\inf_{j\in \Nset_0} c \sbra*{ \varrho^{r-1} \ee^{2^{2j} \bar t} \sqrt \tau +\varrho^{r} 2^{-js(q-1)r^\prime} }
	\end{equation*}
	holds true. Now choosing $j$ such that
	\(
		2^{2j}\sim\frac{1}{\bar t} \ln  \sqrt{3+\frac{\varrho}{\sqrt \tau}}
	\)
	we obtain that
	\begin{align*}
		\psi_\text{vsc}(\tau) &\leq c \varrho^r  \sbra*{\frac{\sqrt \tau}{\varrho} \rbra*{3+\frac{\varrho}{\sqrt \tau}}^{1/2} +\rbra*{\ln \rbra*{\rbra*{3+\frac{\varrho}{\sqrt \tau}}^{1/2}}}^{-s(q-1)r^\prime/2}}\\ 
		& \leq c \varrho^r \rbra*{\ln \rbra*{\frac{\varrho}{\delta}}}^{-s(q-1)r^\prime/2} \sbra*{1+o(1)}\qquad \text{as } \delta\rightarrow 0.
	\end{align*}
Existence of $\fad$ and \eqref{eq:heat_rate} follows along the lines of 
the proof of Theorem \ref{thm:finiteVsc}. 
\end{proof}

\begin{remark}\label{rem:strategyInterchange}
	Note that the verification of \eqref{eq:strategyIllposed} is very different from Theorem \ref{thm:finiteVsc}. For the latter the choice of $P_j$ is not important as long as the Bernstein and Jackson inequalities of Assumption \ref{ass:projectionBesov} hold. Here, however, the forward operator $F$ is too smoothing in order to get Lipschitz stability estimates in the scale of Besov spaces. Therefore, we use an inequality of the form $\norm{P_j (f_1-f_2)}\leq \sigma_j \norm{F (f_1)-F(f_2)}$ for some sequence $\sigma_j>0$ quite similar to Example \ref{ex:strategy}. Such an inequality can in general only be verified for appropriately chosen 
	$P_j$. We refer to \cite{HW:15} for another example with $\gamma\neq 0$.
\end{remark}

\begin{theorem}\label{thm:lowerBH}
	Let $1\leq p \leq 2$, $1 \leq q \leq \infty$, $s,\varrho>0$. Then there cannot exist a reconstruction method 
	$R:L^2\to \Besovpq[0]$ 
	for the backward heat equation such that
	\begin{equation*}
		\sup \cbra*{\altnorm*{f-R(\gobs)}{\Besovpq[0]} \colon \altnorm{f}{B^s_{p,\infty}}\leq \varrho, \altnorm*{T_\BH f-\gobs}{L^2} \leq \delta} = o\rbra*{ \varrho \rbra*{\ln \frac{\varrho}{\delta}}^{-\frac{s}{2}}}
	\end{equation*}
	as $\delta \rightarrow 0$. Hence the convergence rate of Theorem \ref{thm:vscBH} is optimal for $q\geq 2$ as $\delta\rightarrow 0$ up to the value of the constant.
\end{theorem}
\begin{proof}
	Choosing $b_{j,l}=\exp\rbra{-2 \bar t \absval{l}^2}$ the assumptions of Proposition \ref{prop:besovLowerBasic} are fulfilled 
as shown in Example \ref{ex:normBesovLower}\ref{ex:normBesovLower:fourier}. Hence we obtain
	\begin{equation*}
		\omega(\delta, \mathcal K)\geq \max_{k\in \Nset_0} \cbra*{\min \cbra*{2^{-ks} \varrho, \ee^{\bar t 2^{2k}} \delta}}
		\geq \min \cbra*{ \varrho \rbra*{\frac{1}{\bar t} \ln \frac{\varrho}{\delta}}^{-\frac{s}{2}}, \varrho}
	\end{equation*}
	where we have chosen $k\in \Nset_0$ such that $2^{2k}\sim\frac{1}{\bar t} \ln \frac{\varrho}{\delta}$. 
	Hence the claim follows from Lemma \ref{lem:lowerContinuity}.
\end{proof}

\subsection{Statistical convergence rates}
For exponentially ill-posed problems a rather coarse interpolation bound is sufficient: 
\begin{lemma}\label{lem:interpolationBH}
	The operator $T_\BH$ fulfills Assumption \ref{ass:interpolation} with $\beta=\frac{1}{2},\gamma=\frac{1}{2r}$.
\end{lemma}
\begin{proof}
	By K-interpolation there exists a constant $c>0$ such that
	\begin{align*}
		\altnorm*{T_\BH(f_1-f_2)}{B^{d/2}_{p,1}} &\leq c \altnorm*{T_\BH(f_1-f_2)}{B^{0}_{p,2}}^\frac{1}{2} \altnorm*{T_\BH(f_1-f_2)}{B^{d}_{p,q}}^\frac{1}{2}.
	\end{align*}
As $p\leq 2$, the first factor can be bounded by $\altnorm*{T_\BH(f_1-f_2)}{L^2}^\frac{1}{2}$.
	To control the second factor we again use $p\leq 2$ to obtain
	\begin{align*}
		\altnorm*{T_\BH(f_1-f_2)}{B^{d}_{p,q}}^q
		&\leq\sum\nolimits_{j\in \Nset_0} 2^{jdq} \sbra*{\sum\nolimits_{l\in I_j} \exp(-2 \bar t \absval{l}^2) \absval*{\widehat{(f_1-f_2)}(l)}^2}^\frac{q}{2}\\
		&\leq \sum\nolimits_{j\in \Nset_0} 2^{jdq} \exp \rbra*{- \frac{\bar t q}{16 \pi^2} 2^{2j}} \sbra*{\sum\nolimits_{l\in I_j}  \absval*{\widehat{(f_1-f_2)}(l)}^2}^\frac{q}{2}.
	\end{align*}
	As there exists a constant $c>0$ such that
	\begin{equation*}
		2^{jdq} 2^{-jd(\frac{1}{2}-\frac{1}{p})} \exp \rbra*{-  \tfrac{\bar t q}{16 \pi^2} 2^{2j}} \leq c \qquad \text{for all } j\in \Nset
	\end{equation*}
	one obtains that
	\begin{equation*}
		\altnorm*{T_\BH(f_1-f_2)}{B^{d}_{p,q}} \leq c \altnorm[\big]{f_1-f_2}{B^{\frac{d}{2}-\frac{d}{p}}_{2,q}}\leq c \altnorm*{f_1-f_2}{B^{0}_{p,q}}
		\le c\rbra*{C_\Delta^{-1}\Delta_\Rpen(f_1,f_2)}^\frac{1}{r}
	\end{equation*}
	by the embedding properties of Besov spaces \eqref{eq:besovEmbedSmooth4Int} and Assumption \ref{ass:bregmanNormBound}. This completes the proof.
\end{proof}

\begin{theorem}
Assume that \eqref{eq:rand_noise} holds true with $F=T_\BH$ and $\rand$ 
satisfying Assumption \ref{ass:dev_ineq}, and consider the setting  
\eqref{eq:RY} with $p\in(1,2]$ and $q\in [2,\infty)$. 
Moreover, suppose that $\ftrue\in B^s_{p,\infty}$ for some $s>0$ 
with $\altnorm{\ftrue}{B^s_{p,\infty}}\leq \varrho$. 
Then $\fad$ in \eqref{eq:tikhonov} is well-defined almost surely,  
and for $\alpha=\varepsilon/4$ it satisfies the following error bounds
\begin{align*}
	\forall t>0\colon\quad 
	\prob\left[ \altnorm*{\fad-\ftrue}{\Besovpq[0]}\ge c \rbra*{ \varrho \rbra*{\ln \rbra*{\frac{\varrho}{\varepsilon}}}^{-s/2}+t}\right] \le \exp\rbra*{-C_\rand \varepsilon^{-\frac{\mu}{4}} t^{\frac{3q-2}{4}\mu}}
\end{align*}
with a constant $c>0$ independent of $\ftrue$, $\varrho$, and $\varepsilon$. 
\end{theorem}
\begin{proof}
Existence of $\fad$ follows from Proposition \ref{prop:existence}.  
	By Theorem \ref{thm:vscBH} a logarithmic variational source condition of the form of Example \ref{ex:conjugate} holds true. 
	Hence from \eqref{eq:conjugate:logarithmic} one obtain that
	\begin{equation*}
		\varphi_\psi (\tau)= C \varrho^r \rbra*{\ln \rbra*{\frac{\varrho}{\tau}}}^{-\frac{s(q-1) r^\prime}{2}} (1+ o(1))\qquad \text{as } t\rightarrow 0.
	\end{equation*}
	Thus Theorem \ref{thm:whiteNoiseRate}\eqref{thm:whiteNoiseRate:rate} together with Lemma \ref{lem:interpolationBH} gives
	\begin{align*}
		\Delta_\mathcal{R}(\fad,\ftrue) \le C \alpha^{-\frac{3r}{3r-2}} \altnorm*{\varepsilon\rand}{B^{-d/2}_{p^\prime,\infty}}^{\frac{4r}{3r-2}}+C\varphi_\psi(4\alpha).
	\end{align*}
	Choosing $\alpha = \frac{\varepsilon}{4}$ and using \eqref{eq:bregmanNormBound:convex} we get
	\begin{align*}
		\altnorm*{\fad-\ftrue}{\Besovpq[0]}\le C \varepsilon^{\frac{1}{3r-2}}\altnorm*{\rand}{B^{-d/2}_{p^\prime,\infty}}^{\frac{4}{3r-2}} +C\varphi_\psi(\varepsilon)^\frac{1}{r}.
	\end{align*}
	Finally by the deviation inequality given by Assumption \ref{ass:dev_ineq} we find that
	\begin{align*}
		\prob\left[ \altnorm*{\fad-\ftrue}{\Besovpq[0]}\ge C \rbra*{ \varrho \rbra*{\ln \rbra*{\frac{\varrho}{\varepsilon}}}^{-\frac{s(q-1)r'}{2r}}+\varepsilon^{\frac{1}{3r-2}}t'^{\frac{4}{3r-2}}}\right] \le \exp\rbra*{-C_\rand t'^\mu}
	\end{align*}
for all $t'>0$. Substituting $t=\varepsilon^\frac{1}{3r-2}t'^{\frac{4}{3r-2}}$ shows the claim.
\end{proof}
For the optimality of this risk bound in the case $p=q=2$ we refer to \cite{GK:99}.

\section{Discussion}
We have shown optimal rates of convergence for finitely smoothing operators 
for Tikhonov regularization with Besov norm penalties $B^0_{p,q}$ with 
$p\in (1,2]$ and $q\geq 2$ if the true solution belongs to $B^s_{p,\infty}$. 
By Proposition \ref{prop:besovLowerBasic} these rates cannot be improved in 
the deterministic case if we 
restrict to smaller spaces $B^s_{\tilde{p},\tilde{q}}$ with $\tilde{p}\geq p$ and 
$\tilde{q}\in [1,\infty]$. For $p=q=2$ we have shown in \cite{HW:17} that 
$B^s_{2,\infty}$ is the largest space in which Tikhonov regularization achieves the 
rate $O(\delta^{s/(s+a})$, and we conjecture that this is also the case for 
$p<2$. Note that the smaller $p$ is chosen, the larger becomes the smoothness class 
on which optimal rates are achieved. 

For $q<2$ our approach in its present form yields error bounds which are 
most likely suboptimal. This case requires further investigations. 
It would be desirable to get rid of the logarithmic factor in 
Corollary \ref{coro:log} and the following remark for $L^p$ loss functions. 
Moreover, the case $p=q$ is more commonly used and a bit more convenient 
algorithmically. 


\appendix

\section{Properties of Besov spaces}\label{sec:besovSpaces} 
In this appendix we collect some properties of the Besov scale $\Besovpq$.

Besov spaces are a quite general class of spaces, in order to provide some intuition on these spaces we will list some properties and special cases here which can e.g.\ be found in \cite{Triebel2010}. First of all for $1<p,q<\infty$ and $s\in \Rset$ the dual space is given via $(B^s_{p,q})^\prime = B^{-s}_{p^\prime, q^\prime}$. The spaces form scales with respect to the smoothness and summation index, for any $\varepsilon>0$, $s\in \Rset$, $p\in[1,\infty]$ and $1\leq q_1 \leq q_2 \leq \infty$ the embeddings
\begin{equation}\label{eq:besovEmbedFine}
	B^{s+\varepsilon}_{p,\infty} \subset B^{s}_{p,1} \subset B^{s}_{p,q_1}\subset B^{s}_{p,q_2} \subset B^{s}_{p,\infty} \subset B^{s-\varepsilon}_{p,1}
\end{equation}
are continuous. Furthermore one can give up smoothness to gain integrability, to be more precise for $p_1\leq p_2$ and $s_1\geq s_2$ the embedding 
\begin{equation}\label{eq:besovEmbedSmooth4Int}
	B^{s_1}_{p_1,q}\subset B^{s_2}_{p_2,q}
\end{equation}
is continuous if $s_1- \frac{d}{p_1} \geq s_2 - \frac{d}{p_2}$ and compact if $s_1>s_2$ (see \cite[\S 4.3.3 Rem.~1]{Triebel2010}).

The classical Lebesgue spaces $L^p$ for $p\neq 2$ are not Besov spaces, but for $1<p< \infty$ the following inclusions hold true 
with continuous embeddings:
\begin{equation}\label{eq:lesbegueBesov}
	B^0_{p,\min\cbra{2,p}}\subset L^p \subset B^0_{p,\max\cbra{2,p}}
\end{equation}
However, if $s$ is not an integer, then $B^s_{p,p}=W^{s,p}$ the Sobolev spaces with the norm given by (assume for simplicity that $0<s<1$):
\begin{equation*}
	\altnorm*{f}{W^{s,p}}=\altnorm*{f}{L^p}+\rbra*{\int \int \frac{\absval*{f(x)-f(y)}^p}{\absval*{x-y}^{d+sp}} \,\dd y\, \dd x}^{\frac{1}{p}}.
\end{equation*}

An important class of solutions to inverse problems are functions $f$ which are smooth up to jumps (or jumps 
in the  $k$th-derivative). It is well known that such functions are in $W^{d/p-\varepsilon,p}$ (or in 
$W^{k+d/p-\varepsilon,p}$ respectively) for all $\varepsilon>0$, however in this 
scale of spaces such functions do not have a maximal smoothness index. Using the Nikol'skij representation 
of the Besov norm $B^s_{p,\infty}$ (see \cite[Sec.~2.5.12]{Triebel2010})
one can easily calculate that such functions belong to $B^{s}_{p,\infty}$ (or to 
$B^{k+s}_{p,\infty}$ respectively) for $s=d/p$.

For $p,q\in(1,\infty)$ Assumption \ref{ass:bregmanNormBound} is fulfilled via Example \ref{ex:bregmanNormBound}\eqref{ex:bregmanNormBound:convex}, as Besov spaces $\Besovpq$ equipped with the wavelet norm are $\max\cbra{2,p,q}$-convex, see \cite{Kazimierski2013}.

\section{Details on numerical simulation}\label{sec:detailsSimulation}
The true solution is given by the linear spline interpolating the points
\begin{align*}
	\textstyle\cbra*{\rbra*{0,0}, \rbra*{\frac{1}{8},0}, \rbra*{\frac{3}{16},\frac{1}{16}},\rbra*{\frac{1}{4},0},\rbra*{\frac{5}{16},0},\rbra*{\frac{7}{16},\frac{1}{8}},\rbra*{\frac{1}{2},\frac{1}{16}},\rbra*{\frac{5}{8},\frac{3}{16}}, \rbra*{\frac{11}{16},\frac{1}{16}},\rbra*{\frac{3}{4},\frac{1}{8}},\rbra*{\frac{7}{8},0} }
\end{align*}
with periodic boundary condition on $\mathds T^1$.

In order to find a noise vector $\xi$ in \eqref{eq:det_noise} which (at least approximately) maximizes the left hand side of \eqref{eq:det_rate_finite_smoothing} we proceeded as follows: Let $F$ be a compact operator with singular system $(f_j,g_j,\sigma_j)_{j\in \Nset}$ and denote by $f_\alpha$ the minimizer of the Tikhonov functional \eqref{eq:tikhonov_det} for noise free data $\gobs=F(\ftrue)$. One can expect that there exists $c_1>0$ such that
\begin{equation*}
	\sup\nolimits_\xi\altnorm*{\ftrue- \fad}{L^2} \geq c_1 \altnorm*{f_\alpha- \fad}{L^2}.
\end{equation*}
By first order optimality conditions $ \fad-f_\alpha=(F^*F+\alpha I)^{-1}F^* \xi$, therefore the right hand side is of the form
\begin{equation*}
	\altnorm*{f_\alpha- \fad}{L^2}^2=\altnorm*{(F^*F+\alpha I)^{-1}F^* \xi}{L^2}^2=\sum_{j\in \Nset} \absval*{\frac{\sigma_j\,\pairing{\xi}{g_j}}{\sigma_j^2+\alpha} }^2.
\end{equation*}
Since the function $\lambda/(\lambda^2+\alpha)$ is maximized by $\lambda=\sqrt{\alpha}$ the right hand side will be close to its maximum if for $0<c_2<c_3$ we choose $\xi=\sum_{j\in J} \delta_j g_j$ with $\sum_{j\in J}\delta_j^2\leq \delta^2$ where $J$ is such that $c_2 \sqrt\alpha \leq \sigma_j\leq c_3 \sqrt{\alpha}$ for all $j\in J$. 
This leads to the lower bound
\begin{equation*}
	\altnorm*{\ftrue- \fad}{L^2}^2\geq c_1^2 \frac{c_2^2}{1+c_3^2} \frac{\delta^2}{\alpha}
\end{equation*}
on the reconstruction error. Recall that $\frac{\delta^2}{\alpha}$ also appears in the upper bound on the reconstruction error in Proposition \ref{prop:vscRate}, hence for any parameter choice of $\alpha$ for which the upper bound is of optimal order the lower bound will be of the same order, i.e.\ up to constants we will observe the worst case error rate. 
Although for $p\neq 2$ the estimator $\fad$ depends in a non-affine way on $\xi$ we will assume that the same choice of the noise vector will lead to the worst case error.

The operator of the tested example $F=(I-\frac{1}{4\pi^2}\partial_x^2)^{-1}$ on $\mathds T^1$ is compact with singular system
\begin{equation*}
	f_j(x)=g_j(x)=\begin{cases}  \exp\rbra*{2 \pi \ii \frac{j-1}{2} x} & j \text{ odd}\\\exp\rbra*{-2 \pi \ii \frac{j}{2} x} & j \text{ even} \end{cases} \quad \text{and} \quad \sigma_j=\begin{cases} \rbra*{1+\rbra*{\frac{j-1}{2}}^2}^{-1} & j \text{ odd}\\\rbra*{1+\rbra*{\frac{j}{2}}^2}^{-1}  & j \text{ even} \end{cases}
\end{equation*}
with $j\in \Nset$. Therefore we use the approximation $\sigma_j\approx \rbra{\frac{j}{2}}^{-2}$ and apply the above error model with $c_1=1/2$ and $c_2=2$. The values for $\delta_j$ are drawn from a normal distribution and normalized in order to fulfill $\sum_{j\in J}\delta_j^2=\delta^2$.

\bibliographystyle{siamplain}
\bibliography{litVscOptimalBesov}

\end{document}